\documentclass[12pt]{article}
\usepackage{a4wide,color}
\usepackage{amsmath,amsfonts,amssymb,mathrsfs,amsthm,bbm,dsfont}
\usepackage{url}
\usepackage{enumerate} %to alter bulleting and numbering in lists
\usepackage{tikz}

\usepackage{authblk}    %multiple authors and affiliations

\newtheorem{theorem}{Theorem}[section]
\newtheorem{lemma}[theorem]{Lemma}
\newtheorem{proposition}[theorem]{Proposition}

\newtheorem{definition}[theorem]{Definition}
\theoremstyle{definition}

\newtheorem{example}[theorem]{Example}

\usepackage[final]{showkeys}

\newenvironment{remark}% 
  {\par\medbreak\refstepcounter{theorem}%
    \noindent\textbf{Remark~\thetheorem. }}%
  {\qed\par\medskip}

\makeatletter
  \newcommand{\eqnum}{\leavevmode\hfill\refstepcounter{equation}\textup{\tagform@{\theequation}}}
\makeatother

\numberwithin{equation}{section}

\def\EEE{\color{black}}

\def\Q{\mathcal Q}
\def\cA{\mathcal A}
\def\scrE{\mathscr G}
\def\cV{\mathcal V}
\def\D{\mathcal D}

\def\E{\mathcal E}
\def\S{\mathcal S}

\def\Z{\mathcal Z}
\def\M{\mathcal M}

\let\e\varepsilon
\def\cH{\mathscr H}
\def\cL{\mathscr L}
\def\bP{\mathbb P}

\def\R{\mathbb R}
\def\P{\mathcal P}
\def\dual#1#2{\left\langle #1,#2\right\rangle}
\DeclareMathOperator\Prob{Prob}
\def\Expectation{{\mathbb E}}
\def\Indicator{\mathds{1}}
\def\Ent{\mathrm{Ent}}

\renewcommand{\div}{\mathop{\mathrm{div}}}
\DeclareMathOperator\lapl{\Delta\!}

\DeclareMathOperator\grad{\nabla\!}

\DeclareMathOperator*\argmin{argmin}

\DeclareMathOperator\Dom{Dom}

\begin{document}

\title{On the relation between gradient flows and the large-deviation principle, with applications to Markov chains and diffusion}
\author[1,2]{A. Mielke}
\author[3]{M. A. Peletier}
\author[1,*]{D.R.M. Renger}
\affil[1]{Weierstra\ss-Institut, Mohrenstra\ss e 39, 10117 Berlin, Germany}
\affil[2]{Institut f{\"u}r Mathematik, Humboldt-Universit{\"a}t zu Berlin, Rudower Chaussee 25, 12489 Berlin (Adlershof), Germany}
\affil[3]{Department of Mathematics and Computer Science, and Institute for Complex Molecular Systems (ICMS), Eindhoven University of Technology, P.O. Box 513, 5600 MB Eindhoven, the Netherlands}
\affil[*]{Corresponding author; michiel-renger@wias-berlin.de, tel: +49\,30\,20372\,470, fax: +49\,30\,20372\,311}

\date{\today}
\maketitle

\begin{abstract}
Motivated by the occurence in rate functions of time-dependent large-deviation principles, we study a class of non-negative functions $\cL$ that induce a flow, given by $\cL(\rho_t,\dot\rho_t)=0$. We derive necessary and sufficient conditions for the unique existence of a generalized gradient structure for the induced flow, as well as explicit formulas for the corresponding driving entropy and dissipation functional. In particular, we show how these conditions can be given a probabilistic interpretation when $\cL$ is associated to the large deviations of a microscopic particle system. Finally, we illustrate the theory for independent Brownian particles with drift, which leads to the entropy-Wasserstein gradient structure, and for independent Markovian particles on a finite state space, which leads to a previously unknown gradient structure.
\end{abstract}


\section{Introduction}

Gradient flows are an important subclass of evolution
equations. A gradient-flow structure can be exploited to obtain
additional information about the evolution, such as existence and
stability of solutions (see e.g.\ the survey \cite{Ambrosio2008}),
or to derive limiting equations if some system parameter goes to zero
\cite{Sandier2004,Arnrich2012,Miel13?DAEE}. Moreover, gradient flows can provide
additional physical and analytical insight, such as the maximum dissipation of
entropy and energy, or the geometric structure induced by the dissipation distance.

A special subclass is formed by gradient flows with respect to a Wasserstein-type metric. This class was first identified in the seminal work by Jordan, Kinderlehrer, and Otto~\cite{Jordan1998,Otto2001} and later shown to encompass a wide range of known and unknown partial differential equations~\cite{Ambrosio2008,MatthesMcCannSavare09,PortegiesPeletier10,Mielke2011a,LieMie12?GSGC}. For a class of evolution equations defined on discrete spaces such as graphs, various generalizations exist that unite both continuous and discrete spaces into the same structure~\cite{Maas2011,Mielke2012a,Chow2012}.

In many of these Wasserstein gradient flows entropy terms are present in the driving functional, which leads to diffusive behavior in the evolution. The notion of entropy is inherently related to microscopic particle systems, through Boltzmann's entropy formula, or through Sanov's large-deviation theorem. This suggests that a gradient flow that is driven by entropy may also be related to a microscopic particle system through its large deviations. In our earlier work we showed such a connection between large deviations and gradient flows, for a number of systems: the standard Fokker-Planck equation~\cite{Adams2011,Dirr2012,Duong2013a}, a Fokker-Planck equation with decay~\cite{Peletier2013}, a generalized Kramers equation~\cite{Duong2012a} and the Vlasov-Fokker-Planck equation~\cite{Duong2013}. The paper~\cite{Adams2012} and the dissertation~\cite{Renger2013} discuss these and other examples in more detail.

In all these examples, the large-deviation rate functional can somehow be connected to a variational formulation of the limiting, deterministic evolution equation. In some cases this variational formulation is a gradient flow~\cite{Adams2011}; in other cases it can be recognized as a natural generalization of the gradient-flow concept~\cite{Adams2012,Duong2013}; and in yet other cases it appears to have a different structure altogether~\cite{Peletier2013}. 
This all raises the question, that we address in this paper, 

\begin{itemize}
\item[] \emph{Does there exist a general connection between large deviations and (generalizations of) gradient flows?} \eqnum\label{question}
\end{itemize}

Although this question is framed in terms of large deviations, the only ingredient from large deviations that we will actually use below is the variational form of the large-deviation rate function. To reflect this generality, we will use a more abstract point of view, assuming that some source of understanding (large deviations or otherwise) provides a variational formulation of an evolution equation in terms of what we call an \emph{L-function}.

\subsection{L-functions}
%\label{subsec:L-functions}

Let $\Z$ be some continuous state space, and let $\cL(z,s)$ be a function of the state $z\in \Z$ and tangent vector $s$ at $z$ (these concepts are made more precise in Section~\ref{sec:L-functions to generalized flows}). 

\begin{definition}We call~$\cL$ an L-function if 
\begin{enumerate}[(i)]
  \item $\cL\geq0$,
  \item $\cL(z,\,\cdot\,)$ is convex for all $z$,
  \item $\cL$ induces an evolution equation $\dot z_t=\cA_\cL(z_t)$, in the sense that for all $(z,s)$,
\begin{equation}
\label{eq:L-function minimization}
  \cL(z,s)=0  \qquad \iff \qquad s = \cA_\cL(z)
\end{equation}
\end{enumerate}
(writing $t\mapsto z_t$ for a time-parametrized curve in $\Z$).
\label{def:L-function}
\end{definition}

The evolution induced by an L-function $\cL$ is a distinguished subclass of the Lagrangian flow, described by the Euler-Lagrange equation 
\begin{equation}
  D_z\cL(z_t,\dot z_t) - \frac{d}{dt}\big(D_s\cL(z_t,\dot z_t)\big)=0,
\label{eq:Euler-Lagrange}
\end{equation}
where $D_z$ and $D_s$ are derivatives with respect to the first and second variable respectively. For an evolution induced by an L-function, we have that both $D_s\cL(z_t,\dot z_t)=0$ and $D_z\cL(z_t,\dot z_t)=0$, hence \eqref{eq:Euler-Lagrange} is indeed satisfied. Note however, that the induced evolution $\dot z_t=\cA_\cL(z_t)$ is a first-order equation with respect to time, whereas \eqref{eq:Euler-Lagrange} is of second order and has many more solutions.

L-functions and their induced evolutions are abundant; they not only appear in the large deviations of stochastic processes, but are also well known from Hamiltonian mechanics. In Section~\ref{subsec:L-functions from large deviations} we describe how such functions arise from large deviations. 

Much of the theory in this paper requires only the three defining properties of L-functions. Therefore, in the general setting, we look for a connection between L-functions and gradient flows. At the other side of the connection, we will need will need a generalization of the gradient-flow concept, as we explain in the next section.

\subsection{Generalized gradient systems}
\label{subsec:generalized gradient flows}

We first recall the Riemannian description of a gradient flow. For a brief moment, let $\Z$ be a Riemannian manifold with tangent and cotangent spaces $T\Z$ and $T^*\Z$ and symmetric, positive definite metric tensor $G(z):T_z^*\Z\to T_z\Z$. A gradient flow of an entropy functional $\S:\Z\to\R$ in the space $\Z$ is then described by
\begin{equation}
\label{eq:Riemannian gradient flow}
  G(z_t)\dot z_t = -D\S(z_t) \qquad\text{or equivalently}\qquad \dot z_t = -G(z_t)^{-1}D\S(z_t),
\end{equation}
for all $t\in\lbrack0,T\rbrack$. This can be alternatively formulated by defining, for $z\in\Z, s\in T_z\Z$ and $\xi\in T^*_z\Z$,
\begin{align}
\label{eq:quadratic Psi Psi*}
  \Psi(z,s):=\frac12\langle G(z)s,s\rangle &&\text{and its convex dual}&&   \Psi^*(z,\xi):=\frac12\langle \xi,G(z)^{-1}\xi\rangle,
\end{align}
where $\langle\cdot,\cdot\rangle$ denotes the dual pairing between $T_z^*\Z$ and $T_z\Z$. If $G(z)$ is linear and symmetric, then \eqref{eq:Riemannian gradient flow} can be written as
\begin{equation}
\label{eq:Psi* gradient flow}
  D_s\Psi(z_t,\dot z_t) = -D\S(z_t) \qquad\text{or equivalently}\qquad  \dot z_t = D_\xi\Psi^*\big(z_t,-D\S(z_t)\big),
\end{equation}
which, by Fenchel-Young theory, is equivalent to requiring
\begin{equation}
\label{eq:Psi Psi* formulation}
  \Psi(z_t,\dot z_t) + \Psi^*\big(z_t,-D\S(z_t)\big) + \langle D\S(z_t),\dot z_t\rangle=0.
\end{equation}

In the more general setting of this paper, we are after formulations of the form~\eqref{eq:Psi Psi* formulation}, where $\Z$ can be a general linear space, but $\Psi$ and $\Psi^*$ need not be of the quadratic form~\eqref{eq:quadratic Psi Psi*}. Such formulations are also called $\Psi$-$\Psi^*$-structures, entropy-dissipation inequalities or entropy balances; see~\cite{DeMaTo80PEMS, Degi93NPMM, ColVis90CDNE, Luckhaus1995, Mielke2008, Mielke2011, Mielke2011a}. Here, $\Psi$ and $\Psi^*$ need to satisfy a number of conditions in order to yield a meaningful generalized formulation; these are summarized in the following definition.
\begin{definition}
We call the triple $(\Z,\Psi,\S)$ a \emph{generalized gradient system} if $\S:\Z\to\R$, $\Psi:T\Z\to\R$ and for all $z\in\Z$:
\begin{enumerate}
\item[(i)] $\Psi(z,\cdot)$ is convex in the second argument,
\item[(iia)] $\min\Psi(z,\cdot)=0$, and
\item[(iib)] $\Psi(z,0)=0$.
\end{enumerate}
In addition, $\Psi$ is called \emph{symmetric} if $\Psi(z,s)=\Psi(z,-s)$ for all $(z,s)\in T\Z$.

A curve $z:\lbrack0,T\rbrack\to\Z$ satisfying \eqref{eq:Psi Psi* formulation} is then called a flow induced by the gradient system $(\Z,\Psi,\S)$. 
\label{def:gradient system}
\end{definition}

To explain the motivation behind this definition, observe that Condition~\emph{(i)} guarantees the duality between $\Psi$ and its dual $\Psi^*$, i.e.
\begin{align}
  \Psi^*(z,\xi)=\sup_{s\in T_z\Z} \langle \xi,s\rangle - \Psi(z,s), &&\text{and} &&\Psi(z,s)=\sup_{\xi\in T_z^*\Z} \langle \xi,s\rangle - \Psi^*(z,\xi),
\label{eq:Psi-Psi* dual pair}
\end{align}
so that formulations~\eqref{eq:Psi* gradient flow} and \eqref{eq:Psi Psi* formulation} are still equivalent.

With this duality we have the following equivalences for Conditions~\emph{(iia)} and \emph{(iib)}
\begin{subequations}
\begin{align}
  &\min\Psi(z,\cdot)=0                 && \iff && \Psi^*(z,0)=0,            \label{eq:min Psi iff Psi*}\\
  &\Psi(z,0)=0                         && \iff && \min\Psi^*(z,\cdot)=0     \label{eq:Psi iff min Psi*},
\end{align}
\end{subequations}
so that both $\Psi\geq0$ and $\Psi^*\geq0$. This is a central part of the definition; it implies that for curves satisfying~\eqref{eq:Psi Psi* formulation}, we have $\frac{d}{dt}S(z_t)=\langle D\S(z_t),\dot z_t\rangle\leq0$. Moreover, we see that Conditions~\emph{(iia)} and \emph{(iib)} together imply that the dissipation $\Psi(z,s)$ per unit of time is minimal when the generalized velocity~$s$ is zero, and that stationary points of $\S$ are stationary points of the evolution; this can be seen from writing the evolution as $\dot z_t = D_\xi\Psi^*\big(z_t,-D\S(z_t)\big)$.

Finally, concerning the symmetry of $\Psi$: this property is natural in many systems, in which forward and backward velocities lead to the same dissipation (exceptions arise e.g. in plasticity~\cite[Sect.\,3]{Miel03EFME}) or friction~\cite[Eqn.\,(2.2)]{RoSoVo13}). Note that $\Psi$ is symmetric if and only if $\Psi^*$ is symmetric.

The non-negativity, duality and symmetry of $\Psi$ and $\Psi^*$ will play a central role in this paper.

\subsection{From L-functions to generalized gradient flows}

The question~\eqref{question} can now be made more concrete. For each case that we studied, we showed that there is a deeper connection between the L-function and the gradient flow: not only does the zero-level set of $\cL$ generate the evolution $\dot z_t = \cA(z_t)$ (i.e.~\eqref{eq:L-function minimization}), but in fact the {whole} function $\cL$ can be \emph{identified} with the gradient-flow structure for the evolutions, in some exact or approximate sense. This is remarkable, since the function $\cL$ contains much more information than its zero-level set does, and this additional information does not seem relevant for the evolution $\dot z_t = \cA(z_t)$. Inspired by these examples, we make Question~\eqref{question} more precise as
\begin{itemize}
\item[] \emph{For a given L-function $\cL$, can we find a generalized gradient system $(\Z,\Psi,\S)$ that induces the same evolution as $\cL$, and somehow preserves the structure provided by~$\cL$?}
\end{itemize}
%Is there a general connection between L-functions and generalized gradient flows?

Observe that, by Fenchel's inequality, \eqref{eq:Psi Psi* formulation} is always non-negative. Therefore, \eqref{eq:Psi Psi* formulation} describes a  principle similar to~\eqref{eq:L-function minimization}. In Section~\ref{sec:L-functions to generalized flows} we exploit this similarity by constructing, for a large class of L-functions, a generalized gradient system $(\Z,\Psi,\S)$ such that
\begin{equation}
\label{eq:L=Psi Psi* intro}
  \cL(z_t,\dot z_t) = \Psi(z_t,\dot z_t) + \Psi^*\big(z_t,-D\S(z_t)\big) + \langle D\S(z_t),\dot z_t\rangle.
\end{equation}
Once we find such a relation, we can conclude that 
\begin{itemize}
\item the evolution induced by $\cL$ is the generalized gradient flow induced by the system $(\Z,\Psi,\S)$;
\item the cost $\cL$ to deviate from the optimal curve is the same as the `gradient-system defect', described by the right-hand side of \eqref{eq:L=Psi Psi* intro};
\item in this sense, the structure of the generalized gradient system is consistent with $\cL$---in fact it is \emph{equivalent} to $\cL$, in the sense that $\cL$ can be constructed out of $(\Z,\Psi,\S)$, and vice versa.
\end{itemize}

In our first main result, Theorem~\ref{th:L function to generalized gradient system}, we will see the surprising fact that the existence of a generalized gradient system for which \eqref{eq:L=Psi Psi* intro} holds is equivalent to the `integrability condition'
\begin{equation}
\text{there exists }\S:\Z\to\R \text{ such that }  D_s\cL(z,0)=D\S(z) \text{ for all } z.
\label{eq:integrability condition}
\end{equation}
In addition, we will see that the symmetry of the resulting potentials $\Psi$ and $\Psi^*$ is equivalent to the `time-symmetry condition'
\begin{equation}
  \cL(z,s)-\cL(z,-s)=2\langle D\S(z),s\rangle \qquad \text{for all } z,s.
\label{eq:time-symmetric L intro}
\end{equation}
The two conditions~\eqref{eq:integrability condition} and \eqref{eq:time-symmetric L intro} are illustrated graphically in Figures~\ref{fig:integrability condition} and \ref{fig:time-symmetric L-function}. For L-functions from large deviations of empirical processes, as we explain in Section~\ref{subsec:L-functions from large deviations} below, we will interpret both conditions in terms of properties of the underlying stochastic process (Section~\ref{subsec:integrability and time-symmetry}).

%\bigskip

%The typical example to keep in mind \MAP{I find this statement confusing - half of our two examples does not fit this description! And in addition, we felt and feel that the non-quadratic case is more interesting than the quadratic case. Suggestion: remove it here, and give explicit formulas for the L-functions of the two examples in Section 1.5} is a quadratic L-function, defined through a Riemannian metric tensor $G$ by
%\begin{equation}
%\label{eq:quadratic L-function}
%  \cL(z,s):=\frac12\big\langle G(z)\big(s-\cA_\cL(z)\big),s-\cA_\cL(z)\big\rangle,
%\end{equation}
%which clearly induces the evolution $\dot z_t=\cA_\cL(z_t)$. The integrability condition~\eqref{eq:integrability condition} then says that
%\begin{equation*}
%  -G(z)\cA_\cL(z)=D\S(z)
%\end{equation*}
%for some functional $\S$. In that case the induced evolution coincides with the classical Riemannian gradient flow~\eqref{eq:Riemannian gradient flow}. In addition, we see that for a quadratic L-function~\eqref{eq:quadratic L-function} the time-symmetry condition is always satisfied:
%\begin{equation*}
%  \cL(z,s)-\cL(z,-s)=2\langle G(z)\cA(z),s\rangle.% = 2\langle D\S(z),s\rangle
%\end{equation*}
%Indeed, a quadratic L-function~\eqref{eq:quadratic L-function} will yield quadratic potentials of the form~\eqref{eq:quadratic Psi Psi*}, which are always even.

\begin{figure}
\centering
\begin{minipage}{7cm}
\centering
\begin{tikzpicture}[scale=1.2]
  \tikzstyle{every node}=[font=\small]
  \draw[->] (0,0)--(3,0) node[anchor=west]{$s$};
  \draw (1,2)--(1,-1);
  \draw (0,1.3)..controls (0.3,0.9) and (0.7,0.55) ..(1,0.333) ..controls (1.2,0.2) and (1.6,0).. (2,0).. controls (2.4,0) and (2.8,0.3) ..(3,0.5) node[anchor=south]{$\cL(z,s)$};
  \draw (0,1)-- node[anchor=north,sloped,pos=0.20]{$\langle D\S(z),s\rangle$} (3,-1);
  \draw [fill=black] (1,0.333) circle (0.05);
  \draw [fill=black] (2,0) node[anchor=south]{$\cA(z)$} circle (0.05);
\end{tikzpicture}
\caption{integrability condition.}
\label{fig:integrability condition}
\end{minipage}
\qquad
\begin{minipage}{7cm}
\centering
\begin{tikzpicture}[scale=1.2]
  \tikzstyle{every node}=[font=\small]
  \draw[->] (0,0)--(3,0) node[anchor=west]{$s$};
  \draw (1,2)--(1,-1);
  \draw (0,1.7)..controls (0.3,0.9) and (0.7,0.55) ..(1,0.333) ..controls (1.2,0.2) and (1.6,0).. (2,0).. controls (2.4,0) ..(3,0.1) node[anchor=south]{$\cL(z,s)$};
  \draw (0,1)-- node[anchor=north,sloped,pos=0.20]{$\langle D\S(z),s\rangle$} (3,-1);
  \draw [fill=black] (1,0.333) circle (0.05);
  \draw [fill=black] (2,0) node[anchor=south]{$\cA(z)$} circle (0.05);
  \draw[dotted] (1,0.333)--(2.111,2);
\end{tikzpicture}
\caption{time-symmetry condition.}
\label{fig:time-symmetric L-function}
\end{minipage}
\end{figure}

\subsection{L-functions from large deviations}
\label{subsec:L-functions from large deviations}
%The work by Feng and Kurtz recently showed a new, broad, and important source of such formulations that arise from large deviations \cite{Feng2006}.

In Section~3 we study a specific class of L-functions, namely those that arise from large deviations of empirical processes, as we will now explain.

Let $X_1(t),X_2(t),\hdots$ be a sequence of independent Markov processes in a state space $\Omega$, that all have the same linear generator $\Q:\Dom\Q\to C_b(\Omega)$ with $\Dom\Q\subset C_b(\Omega)$, where $C_b(\Omega)$ denotes the space of bounded, continuous functions on $\Omega$. Then, under suitable initial conditions, the empirical process
\begin{equation}
\label{eq:empirical process}
  \rho^{(n)}:t\mapsto\frac1n\sum_{k=1}^n \delta_{X_k(t)}, \qquad t\in[0,T]
\end{equation}
converges almost surely as $n\to\infty$ to the deterministic evolution $\rho:[0,T]\to \mathcal{P}(\Omega)$ satisfying \cite[Th.~11.4.1]{Dudley1989}
\[
  \frac d{dt} \int_\Omega\!\phi(x)\rho_t(dx) = \int_\Omega\! (\Q\phi)(x) \rho_t(dx) \qquad \text{for all } \phi\in \Dom\Q,
\]
which we abbreviate as
\begin{equation}
\label{eq:deterministic evolution}
  \dot \rho_t = \Q^T\rho_t.
\end{equation}

In many cases, the rate of convergence of $\rho^{(n)}$ to the deterministic evolution~\eqref{eq:deterministic evolution} is characterized by a large-deviation principle of the form (see \cite[Def.~1.1]{Feng2006} for the rigorous definition)
\begin{align}
\label{eq:path ldp}
  &\Prob\!\left(\big(\rho^{(n)}_t\big)_{t=0}^T \approx \rho\right) \mathop{\sim}_{n\to\infty} e^{-n \left(I_0(\rho_0)+I_T(\rho)\right)},
  &I_T(\rho):=\int_0^T\!\cL(\rho_t,\dot\rho_t)\,dt,
\end{align}
where $I_0$ characterizes the stochastic fluctuations in the initial datum $\rho^{(n)}_0$, and $I_T$ characterizes the fluctuations in the dynamics of $t\mapsto\rho^{(n)}_t$. The rate functional $I_0+I_T$ can be interpreted as the probabilistic cost, or the `unlikeliness,' to deviate from the deterministic evolution. This interpretation implies formally that $\cL$ is a non-negative function that attains $0$ if and only if \eqref{eq:deterministic evolution} holds. Therefore, $\cL$ induces an evolution $\dot\rho_t=\cA(\rho_t)$ as in \eqref{eq:L-function minimization}, if we define $\cA$ as the adjoint of $\Q$, acting on probability measures. Indeed, such $\cL$ can be expected to be an L-function, as we discuss in more detail in Section~\ref{subsec:large deviations and L-functions}.

L-functions that arise from large deviations of empirical processes
have a very specific structure; this structure can be used to refine
the theory of Section~\ref{sec:L-functions to generalized flows}. In
particular, in our second main result, Theorem~\ref{th:time-rev is detailed balance}, 
we will see that for such L-functions, the
integrability condition~\eqref{eq:integrability condition} and time
symmetry~\eqref{eq:time-symmetric L intro} are directly related to the
detailed balance condition of the underlying stochastic process.

Throughout this work we avoid the term \emph{reversible Markov
  processes}, which has its origin in the microscopic time
reversibility of the paths. In the theory of Markov processes, it is
well-known (and we use this in the proof of Theorem~\ref{th:time-rev is detailed balance})
that this property is equivalent to the
detailed balance condition. The point is that the usage of the term
\emph{reversibility} is quite differently in thermodynamics: the
acronym GENERIC for General Equations for Non-Equilibrium Reversible
Irreverible Coupling is used for combining Hamilitonian systems
(called reversible systems) and gradient systems (called irreversible
systems), see e.g.\ \cite{Mielke2011,Duong2013}. Since ``reversible
Markov processes'' can be described in terms of dissipative gradient
systems, we rather use the term ``time symmetry'', because the
microscopic time symmetry is reflected in the symmetry of the
dissipation potentials.

\subsection{Applications}
\label{subsec:applications}

After we develop the general theory in Section~\ref{sec:L-functions to generalized flows} and \ref{sec:ldp L-functions}, we apply it to two central examples that arise in the large deviations of an empirical process, as explained above.

Our first example, treated in Section~\ref{subsec:Markov stochsys}, concerns a Markov process on a finite state space $\Omega=\{1,\dots,J\}$, with generator matrix $Q$. We consider the empirical process in $\P(\{1,\hdots,J\})$ which converges to the deterministic linear system of first-order equations,
\begin{equation}
\label{eq:Markov evolution}
  \dot \rho_t = Q^T \rho_t, \qquad \text{in } \P(\{1,\hdots,J\})\times[0,T].
\end{equation} 
We prove in the appendix that the corresponding large deviations L-function is given by
\begin{align*}
  \cL(\rho,s)=\sup_{\xi\in\R^J} \xi\cdot s - \cH(\rho,\xi), &&\text{where}&&   \cH(\rho,\xi)= \sum_{i,j=1}^J \rho_i Q_{ij}\left(e^{\xi_j-\xi_i}-1\right)\!.
\end{align*}
From this we construct a generalized gradient system inducing \eqref{eq:Markov evolution} by the methods developed in this paper. 

Our current research was largely motivated by such Markov evolutions, since both \cite{Maas2011, Chow2012} and one of us
\cite{Mielke2012a} serendipitously discovered a (quadratic) entropy-driven
gradient system for these equations, where the metric can be seen as a
finite-space counterpart of the Wasserstein metric. In fact, the
latter gradient structure appears as a special case of the general
gradient structure of nonlinear chemical reactions studied in
\cite{Mielke2011a,MaaMie13?GSRR}. Interestingly, it turns out that 
the generalized gradient system that we derive in this paper is fundamentally
different than the quadratic structure proposed in \cite{Maas2011, Chow2012,Mielke2012a}.

In the second application we study the empirical process of independent Brownian particles in $\Omega=\R^d$ with a drift given by a general force field $F$. The empirical process then converges almost surely to a solution of the drift-diffusion equation,
\begin{equation}
\label{eq:drift-diffusion}
  \dot \rho_t = \Delta \rho_t+\div(\rho_t F), \qquad \text{in } \P(\R^d)\times[0,T].
\end{equation}
The corresponding large deviation L-function is then given by
\begin{equation*}
  \cL(\rho,s)=\frac14\|s-\lapl\rho-\div(\rho F)\|^2_{H^{-1}(\rho)}.
\end{equation*}
Whenever $F$ is itself a gradient, we again apply the theory from this paper to derive a connection with a gradient system for \eqref{eq:drift-diffusion}, and the resulting gradient system is the well-known Wasserstein gradient flow~\cite{Jordan1998,Ambrosio2008}. This connection between large deviations and the Wasserstein gradient system is certainly not new; it is investigated in~\cite{Adams2011,Adams2012} but was probably understood in the probability community well before. We include it here to show how the theory of this paper is a generalization of this result.

\subsection{Relation with earlier work}
%\label{subsec:previous work}

Several examples of this connection have been discussed before, such as in~\cite{Adams2011,Adams2012,Renger2013,Peletier2013,Duong2013a}. In some of this earlier work in this topic, we studied particle systems of the form \eqref{eq:empirical process}, \eqref{eq:deterministic evolution} and \eqref{eq:path ldp}, but for a \emph{fixed time step} $t>0$, which yields a `conditional' large-deviation principle (see \cite[Th.~17]{Peletier2013}, \cite[Prop.~3.2]{Leonard2007}), formally denoted as:
\begin{equation}
  \Prob\big(\rho^{(n)}_t \approx \rho_t \, \big| \,\rho^{(n)}_0 \approx \rho_0\big) \mathop{\sim}_{n\to\infty} \exp\big(\!-n \hat I_t(\rho_t|\rho_0)\big), \label{eq:conditional ldp}
\end{equation}
For the Fokker-Planck equation, we could then prove that, in the sense of a special $\Gamma$-convergence \cite{Adams2011,Duong2013a}:
\begin{equation}
\label{eq:Fokker-Planck small-time development}
  \frac12 \hat I_t(\rho_t|\rho_0) \sim \S(\rho_t)-\S(\rho_0) + \frac1{2t}d^2(\rho_0,\rho_t) \qquad\text{as }t\to0,
\end{equation}
where $\S$ is the free energy with respect to the Lebesgue measure and $d$ is the Wasserstein metric. In the right-hand side of \eqref{eq:Fokker-Planck small-time development} one recognizes the discrete-time approximation of the Wasserstein gradient flow of the free energy, as found in the original paper \cite{Jordan1998}.

Let us also mention here the relation between our time-reversibility condition~\eqref{eq:time-symmetric L intro} and the concept of entropy production as described in \cite{Maes1999,Maes2000}. For the empirical measure~\eqref{eq:empirical process} of a fixed number of particles $n$, they define a notion of entropy $\Ent^{(n)}$ through the relation
\begin{equation}
\label{Maes}
  \Ent^{(n)}(\rho_T)-\Ent^{(n)}(\rho_0) = \log\frac{d\bP^{(n)}_{\lbrack0,T\rbrack}}{d\bP^{(n)}_{\lbrack T,0\rbrack}}(\rho),
\end{equation}
where $\bP^{(n)}_{\lbrack0,T\rbrack}$ is the probability on paths of the empirical measure, and $\bP^{(n)}_{\lbrack T,0\rbrack}$ is the probability on time-reversed paths. This expression presumes a property of $\bP^{(n)}_{\lbrack0,T\rbrack}$, which is that the right-hand side can be written as a function of the initial and final points of $\rho$ only, a property akin to a vector field being conservative.

If we scale this identity with $1/n$, and call $2\S$ the limit of this entropy, i.e.\ formally $2\S:=\lim_{n\to\infty}\frac1n \Ent^{(n)}$, then the large~deviation principle \eqref{eq:path ldp} gives, again formally,
\begin{multline*}
    2\S(\rho_T)-2\S(\rho_0)=\lim_{n\to\infty} \frac1n\Ent^{(n)}(\rho_T)-\lim_{n\to\infty}\frac1n \Ent^{(n)}(\rho_0) \\
  = \lim_{n\to\infty}\frac1n\log\bP^{(n)}_{\lbrack0,T\rbrack}(\rho)-\lim_{n\to\infty}\frac1n\log\bP^{(n)}_{\lbrack T,0\rbrack}(\rho) = \int_0^T\!\cL(\rho_t,\dot\rho_t)\,dt-\int_0^T\!\cL(\rho_t,-\dot\rho_t)\,dt.
\end{multline*}
This is exactly the time-integrated version of our time-symmetry condition~\eqref{eq:time-symmetric L intro}, which we show to be equivalent to detailed balance (Section~\ref{sec:ldp L-functions}). This confirms the relationship between the assumption underlying~\eqref{Maes} and detailed balance that was also discussed in~\cite{Maes2000}.

\medskip
Finally, we like to note that in practical applications, often one does not have a truly Riemannian manifold. In the case of the Wasserstein space, the linear structure of the space of Borel measures can be used to define derivatives, the tangent space is still well-defined~\cite{Ambrosio2008}, and at least formally, a generalized gradient system in the sense of Section~\ref{subsec:generalized gradient flows} can be described. On general metric spaces, however, this may fail. It is then possible to define a purely metric formulation of generalized gradient systems, for which we refer to \cite{Rossi2008}.

\section{From L-functions to generalized gradient flows}
\label{sec:L-functions to generalized flows}

In this section we will see exactly under which conditions a generalized gradient system can be constructed out of a given L-function. To stress the generality of our work, we now switch to a more general setting, not necessarily related to large deviations. Throughout this section, $\Z$ denotes a space with well-defined (linear) tangent vector bundle $T\Z$ and cotangent vector bundle $T^*\Z$ that are in dual pairing with each other via a bilinear form $\langle\cdot,\cdot\rangle$. All Fr{\'e}chet derivatives are defined through this dual pairing: if $\S:\Z\to\R$ and $\Psi:T\Z\to\R$ then $D\S(z),D_s\Psi(z,s)\in T_z^*\Z$, and if $\Psi^*:T^*\Z\to\R$ then $D_\xi\Psi^*(z,\xi)\in T_z\Z$. {To focus on the main ideas, we assume that all functionals are everywhere finite and differentiable (see Remark~\ref{rem:differentiability}).}

\subsection{The relationship between $\cL$ and the generalized gradient system}
%\label{subsec:construction}

Assume we are given an L-function $\cL:T\Z\to\R$. In what
follows, it will be convenient to consider its convex dual $\cH$,
defined through the relations
\begin{align}
\label{eq:LH dual pair}
  \cH(z,\xi)=\sup_{s\in T_z\Z} \langle \xi,s\rangle - \cL(z,s), &&\text{and} &&\cL(z,s)=\sup_{\xi\in T_z^*\Z} \langle \xi,s\rangle - \cH(z,\xi).
\end{align}
%The following lemma describes the construction of $\Psi$ and $\Psi^*$ when the entropy functional~$\S$ is already known. 

The following lemma is central in all that follows, despite its short and simple proof. In this construction, we will allow any covector field $\cV(z)$, not necessarily of the form $D\S(z)$. 
\begin{lemma} Let $\cL$ and $\cH$ be related through \eqref{eq:LH dual pair}, and $z\mapsto \cV(z)\in T_z^*\Z$ be any covector field. Moreover, let $\Psi$ and $\Psi^*$ be related through \eqref{eq:Psi-Psi* dual pair} with $\Psi^*(\cdot,0)=0$. Then $\cL$ can be written as
\begin{equation}
\label{eq:L=Psi Psi*}
  \cL(z,s) = \Psi(z,s) + \Psi^*\big(z,-\cV(z)\big) + \langle \cV(z),s\rangle \qquad \text{for all } (z,s)\in T\Z,
\end{equation}
if and only if $\Psi^*=\Psi^*_{\cL,\cV}$ where
\begin{equation}
  \Psi_{\cL,\cV}^*(z,\xi):=\cH\big(z,\cV(z)+\xi\big)-\cH\big(z,\cV(z)\big).
\label{eq:Psi* from H}
\end{equation}
\label{lem:L=Psi Psi*}
\end{lemma}
\vspace{-0.8cm}
\begin{proof}
First assume that $\Psi^*=\Psi^*_{\cL,\cV}$. We calculate
\begin{align}
  \Psi_{\cL,\cV}(z,s) &=\sup_{\xi}\, \langle \xi,s\rangle - \Psi_{\cL,\cV}^*(z,\xi) \notag\\
                &=\sup_{\xi}\, \langle \xi,s\rangle - \cH\big(z,\cV(z)+\xi\big) + \cH\big(z,\cV(z)\big) \notag\\
                &=\sup_{\xi}\, \langle \xi-\cV(z),s\rangle - \cH\big(z,\xi\big) + \cH\big(z,\cV(z)\big)\notag\\
                &=\cL(z,s) - \langle \cV(z),s\rangle + \cH\big(z,\cV(z)\big) \label{eq:calculation Psi**}.
\end{align}
Since $\cH\big(z,\cV(z)\big) =-\Psi_{\cL,\cV}^*\big(z,-\cV(z)\big)$ by~\eqref{eq:Psi* from H}, we obtain \eqref{eq:L=Psi Psi*}.

For the other direction, assume that \eqref{eq:L=Psi Psi*} holds. The dual of~\eqref{eq:L=Psi Psi*} is then given by
\begin{align*}
  \cH(z,\xi) &= \sup_s\,\langle \xi,s\rangle-\cL(z,s)\\
    &=\sup_s\,\big\langle \xi-\cV(z),s\big\rangle - \Psi(z,s) - \Psi^*\big(z,-\cV(z)\big)\\
    &=\Psi^*\big(z,\xi-\cV(z)\big)- \Psi^*\big(z,-\cV(z)\big).
\end{align*}
The last term can be found by substituting $\xi=\cV(z)$:
\begin{equation*}
  \cH\big(z,\cV(z)\big)=-\Psi^*\big(z,-\cV(z)\big),
\end{equation*}
since $\Psi^*(z,0)=0$ by assumption. We conclude that $\Psi^*$ is of the form $\eqref{eq:Psi* from H}$.
\end{proof}

The surprising aspect of this lemma is that \emph{any} L-function $\cL$ can be written in the form \eqref{eq:L=Psi Psi*}, for \emph{any} covector field $\cV$. Therefore, this property contains no information in itself. Convenient choices of the covector field $\cV$ might be determined by the question whether (a) $\cV$ is a derivative of some functional $\S$, and (b) whether the corresponding $\Psi$ and $\Psi^*$ are both non-negative (see \eqref{eq:min Psi iff Psi*} and \eqref{eq:Psi iff min Psi*})---since the combination of both leads to monotonicity of $\S$. 
We first address property (b) in Proposition~\ref{prop:time-symmetric L to gradient flow} below, and we simultaneously address the question when $\Psi$ is symmetric.

\begin{proposition} Let an L-function $\cL:T\Z\to\R$ and a covector field $\cV$ be given, and let $\Psi_{\cL,\cV}^*$ and its dual $\Psi_{\cL,\cV}$ be constructed according to Lemma~\ref{lem:L=Psi Psi*}. Then 
\begin{enumerate}[(i)]
\item \label{it:locally reversible is pseudo-gradient system} $\Psi_{\cL,\cV},\Psi_{\cL,\cV}^*\geq0$  if and only if $\cV(z)=\cV_\cL(z):=D_s\cL(z,0)$ for all $z\in\Z$. In that case, $\Psi_{\cL,\cV}$ can be written as
\begin{equation}
  \Psi_\cL(z,s):=\cL(z,s)-\cL(z,0)-\langle \cV_\cL(z),s\rangle.
\label{eq:Psi from L}
\end{equation}

\item \label{it:time symmetry is symmetric pseudo-gradient system} Assume that (\ref{it:locally reversible is pseudo-gradient system}) is satisfied. Then the following statements are equivalent:
  \begin{itemize}
  \item $\Psi_\cL$ is symmetric (Definition~\ref{def:gradient system}), %\label{it:symmetric potential}
  \item $\Psi^*_\cL$ is symmetric, i.e. $\Psi^*(z,\xi)=\Psi^*(z,-\xi) \quad\forall(z,\xi)\in T^*\Z$,
  \item $\cL(z,s)-\cL(z,-s)=2\langle \cV_\cL(z),s\rangle \quad\forall(z,s)\in T\Z$, \quad\text{`time symmetry'}    \eqnum\label{eq:time-symmetric L}
  \item $\cH\big(z,\cV_\cL(z)-\xi\big)=\cH\big(z,\cV_\cL(z)+\xi\big) \quad\forall(z,\xi)\in T^*\Z$.                 \eqnum\label{eq:time-symmetric H}
  \end{itemize}
\end{enumerate}
\label{prop:time-symmetric L to gradient flow}
\end{proposition}

\begin{proof}To prove \eqref{it:locally reversible is pseudo-gradient system}, we first assume that $\cV(z)=\cV_\cL(z)=D_s\cL(z,0)$. By duality, $\cV_\cL(z)$ {is} a minimizer of $\cH(z,\cdot)$, so that \eqref{eq:Psi* from H} gives
\begin{equation*}
%\label{eq:inf H=L0}
  \Psi_{\cL,\cV}^*(z,\xi)=\cH\big(z,\cV_\cL(z)+\xi\big)-\cH\big(z,\cV_\cL(z)\big)\geq0,
\end{equation*}
and clearly $\Psi_{\cL,\cV}^*(z,0)=0$. By \eqref{eq:min Psi iff Psi*} this implies that $\Psi_{\cL,\cV},\Psi_{\cL,\cV}^*\geq0$.
Conversely, if $\Psi_{\cL,\cV},\Psi_{\cL,\cV}^*\geq0$, then $D_s\Psi_{\cL,\cV}(z,0)=0$. Differentiating \eqref{eq:L=Psi Psi*} then gives
\begin{equation*}
  D_s\cL(z,0)=D_s\Psi_{\cL,\cV}(z,0)+\cV(z)=\cV(z),
\end{equation*}
and hence $\cV=\cV_\cL$.

To prove \eqref{eq:Psi from L}, observe that $\Psi_\cL(z,0)=0$, and~\eqref{eq:calculation Psi**} gives
\begin{equation*}
  0=\Psi_\cL(z,0)=\cL(z,0) + \cH\big(z,\cV_\cL(z)\big).
\end{equation*}
Identity~\eqref{eq:Psi from L} follows from resubstituting this expression into~\eqref{eq:calculation Psi**}.

\eqref{it:time symmetry is symmetric pseudo-gradient system} The equivalence of the symmetry of $\Psi_\cL$ and $\Psi^*_\cL$ follows directly from the duality relation~\eqref{eq:Psi-Psi* dual pair}. By identity~\eqref{eq:Psi from L}, the symmetry property $\Psi_\cL(z,s) = \Psi_\cL(z,-s)$ is equivalent to~\eqref{eq:time-symmetric L}. By~\eqref{eq:Psi* from H}, condition~\eqref{eq:time-symmetric H} is equivalent to symmetry of $\Psi_\cL^*$.
\end{proof}

\subsection{The first main result}

By Proposition~\ref{prop:time-symmetric L to gradient flow}, the main conditions on $\cL$ to supply the induced evolution $\dot z_t=\cA(z_t)$ with a gradient structure are that the covector field $\cV(z)$ equals $D_s\cL(z,0)$, and that this covector field satisfies the \emph{integrability condition} $D_s\cL(z,0)=D\S(z)$ for some $\S$. This leads to the following abstract result:
\begin{theorem} For any given L-function $\cL$ there holds:
\begin{enumerate}[(i)]
\item If the covector field $D_s\cL(z,0)=D\S_\cL(z)$ for some $\S_\cL:\Z\to\R$, and $\Psi_\cL,\Psi^*_\cL$ are constructed according to Lemma~\ref{lem:L=Psi Psi*} with $\cV = D\S_\cL$, then $(\Z,\Psi_\cL,\S_\cL)$ is a generalized gradient system that induces the evolution $\dot z_t=\cA_\cL(z_t)$ and satisfies~\eqref{eq:L=Psi Psi*}.
\item Moreover, any generalized gradient system $(\Z,\Psi,\S)$ inducing the evolution $\dot z_t=\cA_\cL(z_t)$ and satisfying \eqref{eq:L=Psi Psi*} equals $(\Z,\Psi_\cL,\S_\cL)$, up to adding a constant to $\S$.
\item In addition, $\Psi_\cL$ is symmetric if and only if $\cL$ satisfies the time-symmetry condition~\eqref{eq:time-symmetric L}.
\end{enumerate}
\label{th:L function to generalized gradient system}
\end{theorem}
\begin{proof} The theorem is a direct consequence of Lemma~\ref{lem:L=Psi Psi*} and Proposition~\ref{prop:time-symmetric L to gradient flow}, using the covector field $\cV_\cL=D\S_\cL$.
\end{proof}

We return to the interpretation of this theorem in the discussion in Section~\ref{sec:discussion}.

%\paragraph{Interpretation of Theorem~\ref{th:L function to generalized gradient system}.}
%\begin{enumerate}
%\item 

\begin{remark}
  As discussed in the introduction of this section, we have assumed that $\cL(z,\cdot)$ is differentiable in $0$. Many of these results generalize to the case where $\cL(z,\cdot)$ is convex, lower semicontinuous, and may attain the value $+\infty$. Then the subdifferential $\partial_s\cL(z,0):=\{\xi\in T^*_z\Z:\forall s\in T_z\Z \quad \cL(z,s)\geq\cL(z,0)+\langle\xi,s\rangle\}$ is well-defined, and for instance the condition $D_s\cL(z,0) = D\S(z)$ can be generalized to finding an $\S_\cL$ such that $D\S_\cL(z)\in\partial_s\cL(z,0)$. In some cases $\partial_s\cL(z,0)$ may be empty. This happens for example if $\cL$ corresponds to the large-deviation rate functional of the empirical process on a finite state space $\{1,\hdots,J\}$, where the underlying Markov processes are reducible and the tangent space includes tangents of the form $\mathds1_{\omega_1}-\mathds1_{\omega_2}$ for $\omega_1$ and $\omega_2$ from different irreducible components of $\{1,\hdots,J\}$.
\label{rem:differentiability}
\end{remark}

\section{Large-deviation L-functions of empirical processes}
\label{sec:ldp L-functions}

In this section we decrease the level of abstraction by one notch, and focus on a specific class of L-functions, namely those that arise from a large-deviation principle of an empirical process. Throughout this section, we consider a general Polish (i.e.\ complete separable metric) space $\Omega$, and a general linear generator $\Q:\Dom\Q\to C_b(\Omega), \Dom\Q\subset C_b(\Omega)$. As usual, we assume that $\Q1\equiv0$, which is equivalent to requiring conservation of probability. As described in Section~\ref{subsec:applications}, we study the large deviations of the empirical process $\rho^{(n)}$, defined in~\eqref{eq:empirical process}, corresponding to a sequence of independent Markov processes with generator $\Q$.

We define $\M(\Omega)$ as the space of finite, signed Borel measures on $\Omega$. Since $\rho^{(n)}$ is a (random) probability measure, the state space $\Z$ is a subspace of $\P(\Omega):=\{\rho\in\M(\Omega):\rho\ge0,\,\rho(\Omega)=1\}$. Note that we denote elements of $\Z$ as $\rho$, to emphasize that we are now working with measures. The tangent bundle $T_\rho\Z$ will formally be a subspace of $\M_0(\Omega):=\{s\in\M(\Omega):\rho(\Omega)=0\}$. For general purposes it suffices to take the cotangent bundle $T_\rho^*\Z$ as a subset of $C_b(\Omega)$, with the usual dual pairing
\begin{equation*}
  \langle \xi,s\rangle := \int_\Omega\!\xi(x)\,s(dx).
\end{equation*}
The exact spaces $\Z\subset\P(\Omega), T\Z\subset\Z\times\M_0(\Omega)$ and $T^*\Z\subset\Z\times C_b(\Omega)$ depend on the situation, and must be chosen case-by-case.

\subsection{Large-deviation principles and L-functions}
\label{subsec:large deviations and L-functions}

Today, there are two common methods to derive large-deviation principles of the form \eqref{eq:path ldp}. Both methods yield the convex dual $\cH:T^*\Z\to\R$ of the L-function $\cL$. As we have seen, this duality plays a crucial role in the connection with generalized gradient flows.

The first method, sometimes called Cram{\'e}r's trick, is based on the idea that, in order to calculate the large deviation rate for any path $\rho$, one can perturb the generator in such a way that the curve $\rho$ becomes the most likely one (see for example~\cite[Ch.~10]{Kipnis1999}). As an illustration and for the sake of completeness, we use this method to prove a large-deviation principle for the empirical process of finite-state Markov chains in Section~\ref{subsec:Markov stochsys}. Since the arguments are more-or-less standard, we postpone this proof to the appendix.

The other method, due to Feng and Kurtz \cite{Feng2006}, provides a formal way to calculate the large-deviation rate directly, although it requires more advanced techniques to make the arguments rigorous. It can be used to calculate a broad class of large deviations of a family of Markov processes. For empirical processes of independent particles, the formalism gives us the explicit formula (see \cite[Example~1.14 and Section~13.3]{Feng2006}):
\begin{equation}
\label{eq:H from Feng-Kurtz}
  \cH(\rho,\xi):=\int_\Omega\! e^{-\xi} \big(\Q e^\xi\big)\, d\rho.
\end{equation}
This function also plays an important role in Cram\'er's method; for example, it appears in our large-deviation proof for finite-state Markov chains in equation~\eqref{eq:app nonlinear generator}. %By duality~\eqref{eq:LH dual pair} one finds the transform $\cL$, and obtains a large-deviation principle of the form~\eqref{eq:path ldp}. 

\medskip
By either method one arrives at the large-deviation principle~\eqref{eq:path ldp}, with 
some function $\cL:T\Z\to\R$, related to $\cH$ through
\begin{equation}
  \cH(\rho,\xi)=\sup_{s}\,\langle \xi,s\rangle-\cL(\rho,s).
\label{eq:emp measure H dual of L}
\end{equation}
We point out that for general large deviations, $\cL$ may not be convex (this can occur for example when the limit process is not deterministic). In that case, $\cL$ will not be an L-function in the sense of Definition~\ref{def:L-function} and we can not connect to the theory from Section~\ref{sec:L-functions to generalized flows}; in particular, the other duality relation \eqref{eq:LH dual pair} may not hold. We show here that these problems do not occur for empirical processes.
\begin{proposition} Assume the empirical process \eqref{eq:empirical process} satisfies the large-deviation principle~\eqref{eq:path ldp} in the Skorohod space $\D\big([0,T];\P(\Omega)\big)$, where $\cH$ is given by \eqref{eq:H from Feng-Kurtz}, and $\cL$ is related to $\cH$ through \eqref{eq:emp measure H dual of L}. Then $\cL$ is an L-function, i.e.
\begin{enumerate}[(i)]
\item $\cL\geq0$,
\item $\cL(\rho,\cdot)$ is convex for all $\rho$, hence $\cL(\rho,s)=\sup_\xi\,\langle\xi,s\rangle-\cH(\rho,\xi)$, and 
\item $\cL(\rho,s)=0 \iff s=\Q^T\rho$,
\end{enumerate}
 where $\Q^T$ is the adjoint of the one-particle generator $\Q$, acting on probability measures.
\end{proposition}
\begin{proof}
The non-negativity of $\cL$ follows immediately from the definition of the large-deviation principle.

%To prove the convexity, we first construct the empirical measure on curves
%\begin{equation*}
%  \theta^{(n)}:=\frac1n\sum_{k=1}^n \delta_{\left(X_k(t)\right)_{t=0}^T},
%\end{equation*}
%which lies in $\P\big(\D([0,T];\Omega)\big)$, where $\D([0,T];\Omega)$ is the Skorohod space of curves $x:[0,T]\to\Omega$. Sanov's Theorem~\cite[Th.~6.2.10]{Dembo1998} gives an abstract large-deviation principle for $\theta^{(n)}$ with a convex rate functional $I'$. 
%Since the rate function $I$ in~\eqref{eq:path ldp} is the contraction of $I'$ under a linear transformation and $I$ is convex, a localization argument yields the convexity of $\cL$ in $\dot z$.

The proof of the convexity of $\cL$ is inspired by \cite[Lem.~4.6]{Dawson1987}: first define the empirical measure on the curve space
\begin{equation*}
%\label{eq:empirical measure on curves}
  \theta^{(n)}:=\frac1n\sum_{k=1}^n \delta_{\left(X_k(t)\right)_{t=0}^T},
\end{equation*}
which lies in $\P(\D([0,T];\Omega))$, where $\D([0,T];\Omega)$ is the Skorohod space of curves $x:[0,T]\to\Omega$. Then by Sanov's Theorem~\cite[Th.~6.2.10]{Dembo1998}, the random variable $\theta^{(n)}$ satisfies a large-deviation principle, where the rate functional is the relative entropy on the curve space. We only exploit the convexity of this rate functional. Define the transformation $F:\P\big(\D([0,T];\Omega)\big)\to \D\big([0,T];\P(\Omega)\big)$, with
\begin{equation*}
  F(\theta)(dx):=\big(\theta(\{y\in\D([0,T];\Omega):y_t\in dx\})\big)_{t=0}^T.
\end{equation*}
We can then retrieve the empirical process through $F(\theta^{(n)})=\rho^{(n)}$, and the corresponding large-deviation rate is retrieved by a contraction principle \cite[Th.~4.2.1]{Dembo1998}. By linearity of the transformation $F$ we find that the rate functional $I_0 + I_T$ corresponding to $\rho^{(n)}$  also is convex \cite[Th.~III.32]{Hollander2000}. For any $\rho\in \P(\Omega)$, and any two $s_1,s_2\in T_\rho\P(\Omega)$, take two arbitrary curves $(\rho_t)_{t=0}^T,(\mu_t)_{t=0}^T$ in $\D\big([0,T];\P(\Omega)\big)$ such that $\rho_0=\rho=\mu_0$ and $\dot\rho_0=s_1$, and $\dot\mu_0=s_2$. Applying the convexity of $I_T+I_0$ to these two curves, it then follows from the arbitrariness of $T$ that $\cL$ is convex in the second argument. Naturally, this implies that $\cL$ is the convex dual of $\cH$.

To show that $\cL(\rho,\Q^T\rho)=0$, note that because of the convexity of $\cL$ we can write
\begin{equation}
  \cL(\rho,\Q^T\rho)=\sup_\xi\, \langle\xi,\Q^T\rho\rangle-\cH(\rho,\xi),
\label{eq:L from Feng-Kurtz}
\end{equation}
and that from \eqref{eq:H from Feng-Kurtz}, we have
\begin{equation*}
  \langle D_\xi\cH(\rho,0),\xi\rangle=-\int_\Omega\!\xi\underbrace{(\Q1)}_{=0}\,d\rho + \int_\Omega\!(\Q\xi)\,d\rho=\langle\Q\xi,\rho\rangle=\langle\xi,\Q^T\rho\rangle.
\end{equation*}
Since this identity is exactly the stationarity equation for the
supremum in \eqref{eq:L from Feng-Kurtz}, we know that $\xi=0$ is
a critical point of $\xi \mapsto \langle\Q\xi,\rho\rangle
-\cH(\rho,\xi)$. In fact, $\xi=0$ is a maximizer, because
$\cH$ is convex by \eqref{eq:emp measure H dual of L}. Therefore
\eqref{eq:L from Feng-Kurtz} becomes
\begin{equation*}
  \cL(\rho,\Q^T\rho)=-\cH(\rho,0)\stackrel{\eqref{eq:H from Feng-Kurtz}}{=}0.
\end{equation*}

Finally, for the minimizer $\Q^T\rho$ of $\cL(\rho,\cdot)$ to be
unique, it suffices that $\cL$ is strictly convex. This follows from the fact
that $\cL$ is the dual of $\cH$, and $\cH$ is differentiable in $\xi$ by \eqref{eq:H from Feng-Kurtz}.
\end{proof}

\subsection{The second main result}
\label{subsec:integrability and time-symmetry}

From Section~\ref{sec:L-functions to generalized flows} we know that two conditions on L-functions play a crucial role in constructing a corresponding generalized gradient flow: the integrability condition~\eqref{eq:integrability condition} and the time-symmetry condition~\eqref{eq:time-symmetric L}. We will now see that these two conditions are strongly related to the detailed balance condition of each $X_i(t)$ (and hence of the empirical process~$\rho^{(n)}$). We recall the definition here:

\begin{definition} A Markov process with generator $\Q$ satisfies \emph{detailed balance} with respect to a locally finite measure $\pi$ if for all $\psi,\phi\in \Dom\Q$ such that $\Q(\phi\pi),\Q(\psi\pi)\in L^1(\pi)$,
\begin{equation}
\label{eq:detailed balance self-adjoint}
  \int_\Omega\!\phi(x) (\Q\psi)(x)\,\pi(dx)=\int_\Omega\!\psi(x)(\Q\phi)(x)\,\pi(dx).
\end{equation}
\end{definition}
%We avoid the latter term to distinguish this microscopic notion from our notion of time symmetry, which is macroscopic in nature.
Naturally, this definition immediately implies that $\pi$ is invariant under $\Q$, i.e.
\begin{equation*}
  \int_\Omega\!(\Q\psi)\,d\pi=0 \qquad \text{for all } \psi\in \Dom\Q\/  \text{ such that }\Q\psi\in L^1(\pi).
\end{equation*}

With an invariant measure $\pi$ at hand, one can construct the
relative entropy
\begin{equation}
\label{eq:relative entropy}
  \E_\pi(\rho):=\int_\Omega\!\log\frac{d\rho}{d\pi}(x)\,\rho(dx).
\end{equation}
To avoid technicalities, we assume that the state space $\Z$ is chosen such that $\E_\pi$ is always finite. This excludes reducible processes, and we have $\pi>0$ on $\Z$.

The relative entropy~\eqref{eq:relative entropy} plays an important
role in the theory of large deviations of empirical measures. It
appears as the large-deviation rate of the empirical measure (i.e.\
when no time dependence is involved) of independent, identically
distributed random variables, by Sanov's Theorem. It turns out that
the relative entropy also plays an important role in the connection
between large deviations and generalized gradient flows. In fact, we
prove below that, if detailed balance holds, then $\E_\pi/2$ is
exactly the right choice for the functional driving a generalized
gradient flow, in the sense that both the integrability condition and
the time symmetry condition are satisfied. We also prove the reverse
statement. With detailed balance, we can therefore connect to the
theory set out in Section~\ref{sec:L-functions to generalized flows},
and construct a generalized gradient flow with symmetric
potentials~$\Psi,\Psi^*$. Note that the factor $1/2$ compensates for the factor~$2$ in the time-symmetry relation~\eqref{eq:time-symmetric L}.

We present the second main result in a compact version where the
assumptions are stronger than needed. We see the result as a principle
that holds in general contexts, but needs an extra proof for these
cases. In the next section we give possible extensions.
\begin{theorem}
\label{th:time-rev is detailed balance}Let an empirical process
$\rho^{(n)}$ satisfy a large-deviation principle of the form
\eqref{eq:path ldp} for all $T>0$ with corresponding
L-function~$\cL$. Assume that the invariant measure satisfies
$\pi\in\P(\Omega)$
and that $\cL$ is twice Fr\'echet differentiable near
$(\rho,s)=(\pi,0)$. Then the time symmetry~\eqref{eq:time-symmetric L} holds
with respect to $\cV_\cL = \frac12 D\E_\pi$ if and only if the generator $\Q$ satisfies detailed balance with respect to $\pi$. In particular, this implies that the integrability condition~\eqref{eq:integrability condition} holds with $\S=\frac12\E_\pi$.
\end{theorem}

\noindent
By Theorem~\ref{th:L function to generalized gradient system}, it follows that $\cL$ generates a generalized gradient system driven by the functional $\frac12 \E_\pi$.

\begin{remark}\label{rm:TimeRevDB}
The twice differentiability of $\cL$, which is only used for the direction
``$\Rightarrow$'', can be weakened by working on a suitable 
subspace. This requires a subspace $\scrE \subset C_b(\Omega)$ with the following properties:
\begin{enumerate}
\item $\scrE_1:=\mathrm{span}\big(\scrE \cup \{1\}\big)$ is a
  $C_b$-dense subset of $\Dom\Q$, and $\Q\scrE_1 = \{\Q\phi:
  \phi\in\scrE_1\}$ is a dense subset of $L^1(\pi)$; 
\item For every $\phi \in \scrE$ we have $\int_\Omega \phi\, d\pi =
  0$ and $\Q^T(\phi\pi)$ is a finite measure;
\item For every $\phi,\psi\in \scrE$, the function
\begin{equation}
\label{def:a}
a(\e_1,\e_2,\e_3) := \cL\Bigl((1+\e_1\phi)\pi,\Q^T\bigl[(\e_2\phi+\e_3\psi)\pi\bigr]\Bigr)
\end{equation}
is well-defined and twice continuously differentiable in a
neighborhood of $(0,0,0)$. \EEE
\end{enumerate}
\end{remark}

\begin{proof}
We first prove the reverse statement. 
\emph{``$\Leftarrow$''} Since $\pi\in\P(\Omega)$, detailed balance~\eqref{eq:detailed balance self-adjoint} is equivalent to (see for example~\cite[Prop.~II.5.3]{Liggett1985})
\begin{equation}
  \Prob\!\left((\rho^{(n)}_t)_{t=0}^T \in A\right) = \Prob\!\left((\rho^{(n)}_{t})_{t=0}^T\in \sigma_T A\right)
\label{eq:detailed balance time-reversal formulation}
\end{equation}
for all measurable $A\subset\D([0,T];\P(\Omega))$, where $(\sigma_T\rho)_t:=\rho_{T-t}$ denotes the time-reversal. In the probabilities above the curves $\rho^{(n)}$ are assumed to start, at $t=0$, from empirical measures of points drawn from $\pi$. Therefore, by definition of the large-deviation principle,
\begin{equation*}
  I_0(\rho_0)+I_T(\rho) = I_0(\rho_T)+I_T(\sigma_T\rho)
\end{equation*}
for any curve $(\rho_t)_{t=0}^T$. Moreover, by Sanov's Theorem $I_0=\E_\pi$, so that
\begin{equation*}
  \E_\pi(\rho_0) + \int_0^T\!\cL(\rho_t,\dot\rho_t)\,dt = \E_\pi(\rho_T) + \int_0^T\!\cL(\rho_{T-t},-\dot\rho_{T-t})\,dt.
\end{equation*}
Differentiating with respect to $T$ then yields \eqref{eq:time-symmetric L} with $\cV_\cL=\frac12D\E_\pi$, because $(\rho_t)_{t=0}^T$ was arbitrary.

\emph{``$\Rightarrow$''.} We now assume~\eqref{eq:time-symmetric L}
and prove the detailed-balance property~\eqref{eq:detailed balance self-adjoint}
relying only on the weakened differentiability of Remark \ref{rm:TimeRevDB}.  By the density assumption on $\scrE$ it is sufficient to prove~\eqref{eq:detailed balance self-adjoint} for all $\phi,\psi\in\scrE$; fix such $\phi,\psi\in\scrE$ and the corresponding function $a$ defined in~\eqref{def:a}.

Using the definition of $\E_\pi$, we calculate that $D^2\E_\pi(\pi)[s_1,s_2]=\int_\Omega\!\frac{ds_1}{d\pi}\frac{ds_2}{d\pi}\,d\pi$, and therefore
\[
\int_\Omega\!\psi(\Q\phi)\,d\pi =\int_\Omega\! \phi\, d\Q^T(\psi\pi) = D^2\E_\pi(\pi)[\phi\pi,\Q^T(\psi\pi)].
\]
The property of detailed balance is therefore proved if we show that for all $\phi,\psi\in\scrE$
\begin{equation}
\label{eq:symmetry-R}
D^2\E_\pi(\pi)[\phi\pi,\Q^T(\psi\pi)] = D^2\E_\pi(\pi)[\psi\pi,\Q^T(\phi\pi)].
\end{equation}

\medskip
To prove this, first note that since $\cL\bigl((1+\e_1\phi)\pi,\Q^T[(1+\e_1\phi)\pi]+s\bigr)$ is minimal at $s=0$, we have
\[
D_3 a(\e_1,\e_1,0) = 0.
\]
Differentiating this expression with respect to $\e_1$ we find that
\begin{equation}
\label{eq:a13-a23}
D_{13}a(0,0,0) + D_{23} a(0,0,0) =0.
\end{equation}

Secondly, the time symmetry \eqref{eq:time-symmetric L} with
respect to $\frac12 D\E_\pi$ reads explicitly 
\begin{equation}
\label{eq:LLDEpi}
  \cL(\rho,s)-\cL(\rho,-s)= \langle D\E_\pi(\rho),s\rangle.
\end{equation}
Using appropriate $\rho$ and $s$ we obtain
\[
a(\e_1,0,\e_3) - a(\e_1,0,-\e_3) = 2\bigl\langle D\E_\pi\bigl((1+\e_1\phi)\pi\bigr),\Q^T(\e_3\psi\pi)\bigr\rangle,
\]
and by differentiating with respect to $\e_3$ we find
\[
D_3a(\e_1,0,0) = \bigl\langle D\E_\pi\bigl((1+\e_1\phi)\pi\bigr),\Q^T(\psi\pi)\bigr\rangle.
\]
Differentiating again with respect to $\e_1$ yields
\[
D_{13}a(0,0,0) = D^2\E_\pi(\pi)[\phi\pi,\Q^T(\psi\pi)],
\]
which combines with~\eqref{eq:a13-a23} to give
\[
D_{23} a(0,0,0) = -D^2\E_\pi(\pi)[\phi\pi,\Q^T(\psi\pi)].
\]
This identity implies~\eqref{eq:symmetry-R}, since the left-hand side
is unchanged upon exchanging~$\phi$ and~$\psi$, because
$D_{23}a(0,0,0)=D_{32} a(0,0,0)$. This proves~\eqref{eq:detailed balance self-adjoint}. 

\medskip

Finally, differentiating \eqref{eq:LLDEpi} with respect to $s$
provides the desired integrability condition $D_s\cL(\rho,0)=\tfrac12D\E_\pi(\rho)$. 
\end{proof}

Theorem~\ref{th:time-rev is detailed balance} implies that the entropy functional~\eqref{eq:relative entropy} is the only functional (up to constants) for which $\cL$ is time-symmetric. This entropy functional enters the proof because of Sanov's theorem, and measures the cost to deviate from the equilibrium $\pi$.
%The time-symmetry condition~\eqref{eq:time symmetry} states that the probabilistic cost to move forward in time, minus the cost to move backward in time is twice the change in entropy. Heuristically, this explains why we included the factor $1/2$ in the entropy~\eqref{eq:relative entropy}: to capture the cost of moving forward in time only.

\begin{remark} Using the explicit form of $\cH$ from \eqref{eq:H from Feng-Kurtz} together with Proposition~\ref{prop:time-symmetric L to gradient flow}(ii), we can also see more directly that detailed balance of the generator $\Q$ is equivalent to the time symmetry~\eqref{eq:time-symmetric H}:
\begin{align*}
  \cH\big(\rho,\tfrac12D\E_\pi(\rho)+\xi\big)&=\int_\Omega\! e^{-\frac12\log\frac{d\rho}{d\pi}-\frac12-\xi}\big(\Q e^{\frac12\log\frac{d\rho}{d\pi}+\frac12+\xi}\big)\,d\rho\\
                                 &=\int_\Omega\! e^{\frac12\log\frac{d\rho}{d\pi}-\frac12-\xi}\big(\Q e^{\frac12\log\frac{d\rho}{d\pi}+\frac12+\xi}\big)\,d\pi\\
                                 &\!\mathop{=}^{\eqref{eq:detailed balance self-adjoint}} \int_\Omega\! e^{\frac12\log\frac{d\rho}{d\pi}+\frac12+\xi}\big(\Q e^{\frac12\log\frac{d\rho}{d\pi}-\frac12-\xi}\big)\,d\pi\\
                                 &= \int_\Omega\! e^{-\frac12\log\frac{d\rho}{d\pi}-\frac12+\xi}\big(\Q e^{\frac12\log\frac{d\rho}{d\pi}+\frac12-\xi}\big)\,d\rho
                                  =\cH\big(\rho,\tfrac12D\E_\pi(\rho)-\xi\big),
\end{align*}
where we have used the linearity of $\Q$ to take the constant $e^{-1}$
out of the parenthesis.
\end{remark}

\subsection{Generalizations}
%\label{subsec:generalizations}

\subsubsection{Locally finite invariant measures}

Observe that in \eqref{eq:detailed balance self-adjoint} and \eqref{eq:relative entropy}, the invariant measure $\pi$ does not need to be a probability measure. This relaxation is not uncommon in probability theory (see for example \cite[Sect.~II.5]{Liggett1985}), and can be quite useful when dealing with generators on unbounded domains for which an invariant probability measure does not exist, such as the diffusion equation on $\R^d$. It also implies that the value $\pi(\Omega)$ is not relevant. Indeed, if it is finite, this constant will vanish when one considers entropy differences $\E_\pi(\rho_2)-\E_\pi(\rho_1)$ or derivatives~$D\E_\pi(\rho)$. 

We now show that Theorem~\ref{th:time-rev is detailed balance} can be generalized to locally finite invariant measures, under the condition that the rate functional has compact level sets. This assumption is very common in the theory of large deviations, where it is often called `goodness' of the rate functional.

\begin{proposition} Theorem~\ref{th:time-rev is detailed balance} remains true if the invariant measure $\pi$ is only locally finite, and the rate functional $\rho\mapsto
I_0(\rho_0)+\int_0^T\!\cL(\rho_t,\dot\rho_t)\,dt$ has compact level
sets in the topology of the large-deviation principle.
\label{prop:time-rev is detailed balance, loc fin inv measure}
\end{proposition}

\begin{proof} The forward proof of Theorem~\ref{th:time-rev is detailed balance} remains valid; we only prove the reverse statement here. 

By \cite[Lem.~3.5]{Feng2006} the compact level sets of the rate functional imply that
the sequence $\rho^{(n)}$ is exponentially tight, which allows the use of a time-discrete approximation. In the time-discrete setting, detailed balance with respect to a locally finite invariant measure implies the following discrete time-symmetry property \cite[Prop.~8.2.2]{Renger2013}:
\begin{equation}
  \hat I_t(\rho_t|\rho_0)-\hat I_t(\rho_0|\rho_t) = \E_\pi(\rho_t)-\E_\pi(\rho_0),
\label{eq:discrete time-symmetry}
\end{equation}
where $\hat I_t$ is the conditional large-deviation rate~\eqref{eq:conditional ldp}.

First observe that, due to the contraction principle~\cite[Th.~4.2.1]{Dembo1998}, the sequence $\big(\rho^{(n)}_{t_k}\big)_{k=0}^K$ satisfies a large-deviation principle in $\P(\Omega)^K$ for any subset $\{t_0,\hdots t_K\}$ of $[0,T]$. Moreover, because of the Markov property, there holds
\begin{multline*}
  \Prob\big(\rho^{(n)}_{t_K}\in A_K, \rho^{(n)}_{t_{K-1}}\in A_{K-1},\hdots,\rho^{(n)}_{t_1}\in A_1|\rho^{(n)}_{t_0}=\rho_{t_0}\big)\\
  =\Prob\big(\rho^{(n)}_{t_K}\in A_K|\rho^{(n)}_{t_{K-1}}\in A_{K-1})\hdots\Prob(\rho^{(n)}_{t_1}\in A_1|\rho^{(n)}_{t_0}=\rho_{t_0}\big).
\end{multline*}
By definition of the large-deviation principle, the rate functional corresponding to the sequence $\big(\rho^{(n)}_{t_k}\big)_{k=0}^K$ can therefore be written as
$\sum_{k=1}^K \hat I_{t_k-t_{k-1}}(\rho_{t_k}|\rho_{t_{k-1}})$.

It then follows from \eqref{eq:discrete time-symmetry}, the exponential tightness, and \cite[Th.~4.28]{Feng2006} that the pathwise large-deviation rate can be written as
\begin{align*}
  I_T(\rho) &= \sup_{0=t_0<t_1\hdots<t_K=T} \sum_{k=1}^K \hat I_{t_k-t_{k-1}}(\rho_{t_k}|\rho_{t_{k-1}}) \notag \\
                     &= \sup_{0=t_0<t_1\hdots<t_K=T} \sum_{k=1}^K \hat I_{t_k-t_{k-1}}(\rho_{t_{k-1}}|\rho_{t_k}) + \E_\pi(\rho_{t_k})-\E_\pi(\rho_{t_{k-1}}) \notag \displaybreak[0]\\
                     &= \E_\pi(\rho_T)-\E_\pi(\rho_0) + \sup_{0=t_0<t_1\hdots<t_K=T} \sum_{k=1}^K \hat I_{t_k-t_{k-1}}(\rho_{t_{k-1}}|\rho_{t_k})\notag \\
                     &= \E_\pi(\rho_T)-\E_\pi(\rho_0) + I_T(\sigma_T\rho), %\label{eq:detailed balance pathwise reversal}
\end{align*}
where $\sigma_T$ is again the time-reversal operator. Differentiation with respect to $T$ then gives~\eqref{eq:time-symmetric L}.
\end{proof}

\subsubsection{General large-deviation principles}

A large part of the theory developed above remains true if we consider more general sequences of random variables, that is, not necessarily empirical measures. Let $Z^{(n)}_t$ be a sequence of stochastic processes with generator $\Q^{(n)}$, in a state space $\Z$, with a formal differentiable structure as in Section~\ref{sec:L-functions to generalized flows}. (Note that in case of the empirical process as before, the $\Q^{(n)}$ plays the role of the generator of the entire empirical process).
Assume that 
\begin{enumerate}
\item $\Q^{(n)}$ satisfies detailed balance with respect to an invariant measure $\pi^{(n)}\in\P(\Z)$,
\item $Z^{(n)}$ converges narrowly (pointwise in $t$) to a deterministic curve, satisfying the limiting equation
  \begin{equation*}
    \dot z_t=\cA(z_t),
  \end{equation*}
\item For all $T>0$, $Z^{(n)}$ satisfies a large-deviation principle of the form
  \begin{align}
  \label{eq:path ldp general process}
    &\Prob\!\left(\big(Z^{(n)}_t\big)_{t=0}^T \approx z\right) \mathop{\sim}_{n\to\infty} e^{-n \left(I_0(z_0)+I_T(z)\right)},
    &I_T(z):=\int_0^T\!\cL(z_t,\dot z_t)\,dt,
  \end{align}
  where $I_0$ is the large-deviation rate functional of $Z^{(n)}_0$.
\item The sequence of random variables in $\Z$ with law $\pi^{(n)}$ satisfies a large-deviation principle with rate $\E$.
\end{enumerate}
We then have the following generalization of Theorem~\ref{th:time-rev is detailed balance}:
\begin{proposition}
Under the assumptions 1--4 above, the time symmetry~\eqref{eq:time-symmetric L} holds
with respect to $\cV_\cL = \frac12 D\E$, and the integrability condition~\eqref{eq:integrability condition} holds with $\S=\frac12\E$.
\end{proposition}

\noindent 
As before, a consequence is that $\cL$ generates a generalized gradient system driven by the functional $\tfrac12\E$ (if $\cL$ is an L-function; see Remark~\ref{rem:general ldp L-function} below).

\begin{proof}
  The proof is analoguous to the proof of ``$\Leftarrow$'' in Theorem~\ref{th:time-rev is detailed balance}, where we replace 
  \eqref{eq:detailed balance time-reversal formulation} by
\begin{equation*}
  \Prob\!\left((Z^{(n)}_t)_{t=0}^T \in A\right) = \Prob\!\left(\big(Z^{(n)}_t\big)_{t=0}^T\in \sigma_T A\right)\!,
\end{equation*}
and $\E_\pi$ by $\E$.
\end{proof}

\begin{remark} In most applications, the $\cL$ from \eqref{eq:path ldp general process} will again be an L-function, as can be seen as follows.
\begin{enumerate}[(i)]
\item The non-negativity of $\cL$ follows from the definition of the large-deviation principle.
\item Since rate functions can be assumed to be lower semicontinuous, in many cases this forces the dependence of $\cL$ on $\dot z$ to be convex (see e.g.~\cite[Ch.~4]{Giusti2003}, \cite[Th.~13.1.3]{Attouch2006} and \cite[Th.~3.15]{Dacorogna2000}). 
\item By the Portmanteau~Theorem, the pointwise narrow convergence of $Z^{(n)}_t$ is equivalent to $\liminf_{n\to\infty}\Prob\big(Z^{(n)}_t\in U\big)\geq 1$ for all narrowly open $U\subset\Z$ containing the deterministic solution at time $t$. Then for any $\epsilon>0$ there is a $N\geq1$ such that for all $n\geq N$
\begin{equation*}
  0\leq \frac1n\log\Prob\big(Z^{(n)}_t\in U\big)\leq\frac1n\log(1-\epsilon)\to0.
\end{equation*}
Hence for any fixed $t$, the conditional rate functional $\hat I_t$
satisfies $\hat I_t\big(z_t|z_0)=0$ whenever $z_t$ solves $\dot z_\tau=\cA(z_\tau)$ up to time $\tau=t$. Therefore, by a similar time-discretization argument as in the proof of Proposition~\ref{prop:time-rev is detailed balance, loc fin inv measure}, this gives $\cL\big(z_t,\cA(z_t)\big)=0$.

In general, the uniqueness of the minimizer $\cA(z)$ of $\cL(z,\cdot)$ does not follow from the large-deviation principle~\eqref{eq:path ldp general process}. In practice, this is often related to finding the right order of $n$ for which the large-deviation $\cL$ does have a unique minimizer.
\end{enumerate}
\label{rem:general ldp L-function}
\end{remark}

\section{Applications}
\label{sec:applications}

One of the questions that led to this study concerns the connection
between large deviations of finite-state-space Markov processes and
the entropy-driven Riemannian gradient structure uncovered in
\cite{Maas2011,Chow2012,Mielke2012a}. As we described in the
introduction, this structure is a generalization of the Wasserstein
gradient-flow structure in continuous spaces, and therefore the
question naturally arises whether this structure connects to the large
deviations of the empirical process of i.i.d. copies of the underlying
Markov process, following the line of Section~\ref{sec:L-functions to generalized flows}.
Below we calculate these large deviations, and
interestingly these do \emph{not} generate the above-mentioned
Riemannian gradient structure, although the driving functional is a
multiple of the relative entropy in both cases. However, by the
arguments in Sections~\ref{sec:L-functions to generalized flows}
and~\ref{sec:ldp L-functions} they do generate a well-defined and
symmetric generalized gradient structure, as we show in the first
application.

In the second application we consider the diffusion-drift equation. The connection between the large deviations of the empirical measure and the Wasserstein gradient flow of the free energy was already made in \cite{Adams2011,Duong2013a} for the discrete-time setting, and in \cite[Sect.~4.2]{Adams2012} in the (continuous-time) L-function setting. Nevertheless, we discuss below how the theory developed in Sections~\ref{sec:L-functions to generalized flows} and \ref{sec:ldp L-functions} can be used to reproduce this result.

\subsection{Finite-state Markov chains}
\label{subsec:Markov stochsys}

We now apply the theory of Sections~\ref{sec:L-functions to generalized flows} and \ref{sec:ldp L-functions} to continuous-time Markov chains on a finite state space $\Omega=\{1,\hdots,J\}$. In this case we can identify
\begin{align*}
  \Z:=\P(\{1,\hdots,J\})=\bigg\{\rho\in\R^J_+:\sum_{j=1}^J \rho_j =1\bigg\},
\end{align*}
and the generator $\Q:\R^J\to \R^J$ can be identified with a Markov intensity matrix $Q\in\R^{J\times J}$, i.e.
\begin{align}
  Q_{ij}\geq0\qquad \text{for } i\neq j, &&\text{and}&& \sum_{j=1}^J Q_{ij}=0.
\label{eq:Markov matrix}
\end{align}
The natural tangent and cotangent spaces to $\P(\{1,\hdots,J\})$ are
\begin{align*}
  T_\rho\Z=T^*_\rho\Z:=\bigg\{s\in\R^J:\sum_{i=1}^J s_i=0\bigg\},
\end{align*}
with the usual Euclidean inner product $\langle s,\xi\rangle:=s\cdot\xi$.

\paragraph{Particle system and large deviations.}
We now take independent Markov processes $X_1(t),X_2(t),\hdots$ in $\{1,\hdots,J\}$ with generator matrix $Q$. By identifying probability measures on $\{1,\hdots,J\}$ with vectors in $\R^J$, the empirical process~\eqref{eq:empirical process} is replaced by
\begin{equation*}
  t\mapsto \rho^{(n)}_t:=\frac1n\sum_{k=1}^n \Indicator_{X_k(t)},
\end{equation*}
which is a random variable in $\P(\{1,\hdots,J\})$ that counts the weighted number of particles at each site. Under suitable initial conditions, this process converges almost surely to the solution of \eqref{eq:Markov evolution}. We prove in Theorem~\ref{th:Markov ldp} in the Appendix that a corresponding large-deviation principle~\eqref{eq:path ldp} holds in $BV\!\big(0,T;\P(\{1,\hdots,J\})\big)$ with the weak-\textasteriskcentered\ topology, with
\begin{align}
\label{eq:Markov Lagrangian}
  \cL(\rho,s)=\sup_{\xi\in T_\rho^*\Z} \xi\cdot s - \cH(\rho,\xi), &&\text{and}&&   \cH(\rho,\xi)= \sum_{i,j=1}^J \rho_i Q_{ij}\left(e^{\xi_j-\xi_i}-1\right)\!.
\end{align}

\paragraph{Towards generalized gradient flows.} Let $\pi\in\P(\Omega)$ be invariant, i.e.\ $Q^T\pi=0$. We implicitly assume that $\pi$ has only non-zero coordinates\ {If $\pi_i=0$ for some $i$, then the space $\Z$ needs to be restricted to those $\rho$ for which $\rho_i=0$.}. The relative entropy functional~\eqref{eq:relative entropy} then becomes
\begin{equation*}
  \E_\pi(\rho)=\sum_{i=1}^J \rho_i \log \frac{\rho_i}{\pi_i}.
\end{equation*}
For this application, the role of the detailed-balance condition will appear explicitly in the calculations; let us not assume this condition for now. Following the construction from Lemma~\ref{lem:L=Psi Psi*}, we define (taking $\S=\tfrac12\E_\pi$)
\begin{align}
  \Psi_\cL^*(\rho,\xi) &:=\cH\big(\rho,\tfrac12D\E_\pi(\rho)+\xi\big)-\cH\big(\rho,\tfrac12D\E_\pi(\rho)\big) \notag \\
    &=\sum_{i,j=1}^J \rho_i Q_{ij}\left\lbrack \exp\!\left(\frac12\log\frac{\rho_j}{\pi_j}+\xi_j-\frac12\log\frac{\rho_i}{\pi_i}-\xi_i\right)-\exp\!\left(\frac12\log\frac{\rho_j}{\pi_j}-\frac12\log\frac{\rho_i}{\pi_i}\right) \right\rbrack \notag \\
    &=\sum_{i,j=1}^J \sqrt{\rho_i\rho_j \frac{\pi_i}{\pi_j} } Q_{ij}  \left(e^{\xi_j-\xi_i}-1\right)\!. \label{eq:Markov Psi* form 1}
\end{align}
Under a condition of \emph{weak reversibility} (cf.\ \cite{Mielke2011a}),
\begin{equation}
  Q_{ij}>0 \iff Q_{ji}>0 \qquad \text{for all } i,j=1,\hdots J,
\label{eq:weak detailed balance}
\end{equation}
expression~\eqref{eq:Markov Psi* form 1} can be written as:
\begin{align}
  \Psi_\cL^*(\rho,\xi) &=\frac12 \sum_{i,j=1}^J\left\lbrack \sqrt{\rho_i\rho_j \frac{\pi_i}{\pi_j} } Q_{ij}  \left(e^{\xi_j-\xi_i}-1\right)+\sqrt{\rho_i\rho_j \frac{\pi_j}{\pi_i} } Q_{ji}  \left(e^{\xi_i-\xi_j}-1\right)\right\rbrack\! \notag \\
      &=\frac12 \sum_{i,j=1}^J \sqrt{\rho_i\rho_j Q_{ij}Q_{ji}} \bigg(\cosh\!\Big(\xi_j-\xi_i + \frac12\underbrace{\log\frac{Q_{ij}\pi_i}{Q_{ji}\pi_j}}_{=0 \text{ if d.b. holds}}\Big)-1\bigg). \label{eq:Markov Psi* cosh}
\end{align}
From this form it is immediate that $\inf_\xi\Psi_\cL^*(z,\xi)=0=\Psi_\cL^*(z,0)$ whenever the detailed balance condition~\eqref{eq:detailed balance self-adjoint} holds. By \eqref{eq:min Psi iff Psi*} and \eqref{eq:Psi iff min Psi*}, this implies that $\Psi,\Psi^*\geq0$. Therefore, under the condition of detailed balance, the linear system~\eqref{eq:Markov evolution} is described by the symmetric generalized gradient system $(\P(\{1,\hdots,J\}),\Psi_\cL,\tfrac12\E_\pi)$, as prescribed by Theorem~\ref{th:time-rev is detailed balance}.

\begin{remark}
  In general, the explicit formula for $\Psi_\cL$ can become much more involved than the formula for $\Psi_\cL^*$. For example, $\Psi_\cL^*$ is linear in $Q$, while $\Psi_\cL$ is not.
\end{remark}

To illustrate the principle of generating a generalized gradient structure out of the large deviations, we describe two examples; one for which detailed balance holds, and one for which it does not.

\begin{example}[finite-state Markov chain with detailed balance] For this example we take $J=2$; then detailed balance~\eqref{eq:detailed balance self-adjoint} always holds whenever \eqref{eq:weak detailed balance} holds. We now write $s=(s_1,-s_1)$ and $\xi=(\xi_1,-\xi_1)$. Then \eqref{eq:Markov Psi* cosh} can be written as
\begin{equation*}
  \Psi_\cL^*(\rho,\xi)=\sqrt{\rho_1\rho_2 Q_{12}Q_{21}}\big(\cosh(2\xi_1)-1\big),
\end{equation*}
and its dual becomes
\begin{equation*}
  \Psi_\cL(\rho,s)=\sqrt{\rho_1\rho_2 Q_{12}Q_{21}}\bigg(\cosh^*\!\Big(\frac{s_1}{2\sqrt{\rho_1\rho_2 Q_{12}Q_{21}}}\Big)+1\bigg),
\end{equation*}
where $\cosh^*$ is the Legendre transform of the hyperbolic cosine. By Theorems~\ref{th:L function to generalized gradient system} and \ref{th:time-rev is detailed balance}, we know that the entropy functional $\S_\cL$ has to be the relative entropy~\eqref{eq:relative entropy}, and we find
\begin{equation*}
  \cL(\rho,s)=\Psi_\cL(\rho,s) + \Psi^*_\cL\big(\rho,-\tfrac12D\E_\pi(\rho)\big) + \tfrac12D\E_\pi(\rho)\cdot s,
\end{equation*}
where the left-hand side is the large-deviation function, and the right-hand side describes a generalized gradient system. Note especially that $\Psi_\cL$ and $\Psi_\cL^*$ are indeed symmetric, but not quadratic in $s$ and $\xi$.
\end{example}

\begin{example}[finite-state Markov chain without detailed balance] We now take $J=3$ and
\begin{equation*}
  Q:=\begin{pmatrix} -1  &  1     & 0 \\
                      0  &  -1    & 1 \\
                      1  &  0     & -1
     \end{pmatrix}\!.
\end{equation*}
The invariant measure is $\pi=(1/3,1/3,1/3)$, but detailed balance does not hold. Therefore, we can no longer apply Theorem~\ref{th:time-rev is detailed balance}, so that we can not tell a priori whether the integrability condition~\eqref{eq:integrability condition} holds.

Therefore, we need to go one step back to Proposition~\ref{prop:time-symmetric L to gradient flow}.
We have the explicit formula
$ %\begin{equation*}
  \cH(\rho,\xi)\stackrel{\eqref{eq:Markov Lagrangian}}{=}\rho_1 e^{\xi_2-\xi_1} + \rho_2e^{\xi_3-\xi_2} + \rho_3 e^{\xi_1-\xi_3} -1
$ %\end{equation*}
and find  
\begin{align*}
  \cV_\cL(\rho) &:=D_s\cL(\rho,0)=\argmin \cH(\rho,\cdot)=\frac13 \begin{pmatrix} \log\big(\rho_1/\rho_3\big)\\
                                                                                  \log\big(\rho_2/\rho_1\big)\\
                                                                                  \log\big(\rho_3/\rho_2\big)
                                                                  \end{pmatrix}\!.
\end{align*}
By Proposition~\ref{prop:time-symmetric L to gradient flow}\eqref{it:locally reversible is pseudo-gradient system}, we can still use this covector field $\cV_\cL$ to construct $\Psi_\cL,\Psi_\cL^*$ such that:
\begin{equation*}
  \cL(\rho,s)=\Psi_\cL(\rho,s) + \Psi^*_\cL\big(\rho,-\cV_\cL(\rho)\big) + \cV_\cL(\rho)\cdot s.
\end{equation*}
The right-hand side of this equation however, does not represent a generalized gradient system, since the vector field $\cV_\cL(\rho)$ is non-conservative, i.e. the integrability condition~\eqref{eq:integrability condition} breaks down.
%\label{ex:finite-state Markov no detailed balance}
\end{example}

\subsection{Diffusion with drift}
%\label{subsec:diffusion}

In this application, we study the drift-diffusion equation~\eqref{eq:drift-diffusion}, and derive the Wasserstein gradient flow from \cite{Jordan1998, Otto2001}. The arguments that follow below motivate the choice of the spaces:
\begin{equation*}
  \Z:=\left\{\rho\in W^{1,1}(\R^d): \int_{\R^d}\!\rho=1 \text{ and } \int_{\R^d}\!\Big(x^2\rho + \frac{|\nabla\rho|^2}{\rho} + |F|^2\rho\Big)<\infty\right\},
\end{equation*}
with formal cotangent and tangent bundles (see \cite[Sec.~12.4]{Ambrosio2008} for the rigorous notions):
\begin{align*}
  &T^*_\rho\Z:=\Bigg\{\xi\in H^1(\rho): \int_{\R^d}\!\xi(x)\,dx=0\bigg\}, \text{ and}\\
  &T_\rho\Z:=H^{-1}(\rho)=\bigg\{s\in \mathcal D'(\R^d): \exists v\in \overline{\{\grad\phi:\phi\in C_c^\infty(\R^d)\}}^{L^2(\rho)} \text{ such that } s+\div(\rho v)=0\bigg\}\!,
\end{align*}
and with dual pairing
\begin{equation*}
  \langle \xi,s\rangle := \int_{\R^d}\!\rho v\cdot\nabla \xi
  \qquad\text{if }s+\div(\rho v) =0,
\end{equation*}
where each occurence of $s+\div(\rho v)=0$ is to be interpreted in distributional sense. Naturally, for the generator we take $\Q:\Dom\Q\to C_b(\R^d)$, with $\Dom\Q\supset C_b^2(\R^d)$ and
\begin{equation}
\label{eq:diffusion generator}
  (\Q\phi)(x):=\lapl\phi(x)-F(x)\cdot\grad\phi(x).
\end{equation}
Wherever needed, we tacitly assume sufficient regularity on covector functions $\xi\in H^1(\rho)$ such that all derivatives exist in a weak sense.

\paragraph{Particle system and large deviations.}
We now take a system of independent processes $X_1(t),X_2(t),\hdots$ in $\R^d$, each with generator $\Q$ as described above. The convergence of the empirical process~\eqref{eq:empirical process} and its corresponding large deviations are well-known results (see for example~\cite{Dawson1987}). By formula~\eqref{eq:H from Feng-Kurtz},
\begin{equation}
\label{eq:diffusion H-function}
  \cH(\rho,\xi):=\int_{\R^d}\! e^{-\xi} \big(Qe^\xi\big)\, \rho =\int_{\R^d}\!\left(\lapl\xi+|\nabla\xi|^2-F\cdot\grad\xi\right) \rho.
\end{equation}
The corresponding large-deviation L-function is found through
\begin{align}
  \cL(\rho,s) &=\sup_{\xi\in T^*_\rho\Z} \langle \xi,s\rangle - \cH(\rho,\xi) \notag \\
              &=\sup_{\xi\in T^*_\rho\Z} \langle \xi,s-\lapl\rho-\div(\rho F) \rangle -\int_{\R^d}\!|\nabla\xi|^2\,\rho \notag \\
              &=: \frac14\|s-\lapl\rho-\div(\rho F)\|_{H^{-1}(\rho)}^2. \label{eq:diffusion L-function}
\end{align}
Indeed by \cite[Th~4.5]{Dawson1987}, the empirical process $\rho^{(n)}$ satisfies the large-deviation principle~\eqref{eq:path ldp} in the space of continuous functions with values in $\P(\R^d)$, equipped with the narrow topology, with L-function~\eqref{eq:diffusion L-function}.

\paragraph{Towards gradient flows.} Again, following the construction of Proposition~\ref{prop:time-symmetric L to gradient flow}, we need to find $\cV_\cL(\rho) :=D_s\cL(\rho,0)=\argmin \cH(\rho,\cdot)$. To this aim, we fix a $\rho\in\Z$ and set, for all $\xi\in T_\rho^*\Z$, for now neglecting the dependence of $\cV_\cL$ on $\rho$,
\begin{equation*}
  0 =\langle D_\xi \cH(\rho,\cV_\cL),\xi\rangle=\int_{\R^d}\!(\lapl \xi + 2\grad\xi\cdot\grad\cV_\cL - F\grad\xi)\,\rho,
\end{equation*}
which implies
\begin{equation}
\label{eq:diffusion Lax-Milgram}
  \int_{\R^d}\!\grad\xi\cdot\grad\cV_\cL\,\rho=\frac12\int_{\R^d}\!\grad\xi\cdot\Big(\frac{\grad\rho}{\rho} + F\Big)\,\rho.
\end{equation}
The left-hand side is an inner product on $H^1(\rho)$; the right-hand side is a linear functional in $\xi\in H^1(\rho)$, bounded by definition of the space $\Z$. Therefore, Riesz's Theorem gives us the unique existence of $\cV_\cL(\rho)\in T^*_\rho\Z$.

With this covector field at hand, we use \eqref{eq:Psi* from H} to construct
\begin{align*}
  \Psi^*_\cL(\rho,\xi)  &:= \cH\big(\rho,\cV_\cL(\rho)+\xi\big) - \cH\big(\rho,\cV_\cL(\rho)\big)\\
                        &\!\!\stackrel{\eqref{eq:diffusion H-function}}{=} \int_{\R^d}\!\left(\lapl\xi+|\nabla\xi|^2+2\grad\xi\cdot\grad\cV_\cL(\rho)-F\cdot\grad\xi\right) \rho\\
                        &\!\!\stackrel{\eqref{eq:diffusion Lax-Milgram}}{=} \int_{\R^d}\!|\nabla\xi|^2\,\rho,
\intertext{and}
  \Psi_\cL(\rho,s) &:= \sup_{\xi\in T^*_\rho\Z}\, \langle \xi,s\rangle - \int_{\R^d}\!|\nabla\xi|^2\,\rho = \frac14\|s\|_{H^{-1}(\rho)}^2.
\end{align*}
Up to now, we have not assumed anything about the force field $F$. We see that even in the most general case, the construction yields dissipations potentials $\Psi_\cL,\Psi_\cL^*$ that are symmetric (and even quadratic), and independent of $F$. In general however, the integrability condition $\cV_\cL(\rho)=D\S(\rho)$ is likely to fail, and the best we can do is to split $\cL$ into
\begin{equation*}
  \cL(\rho,s)=\frac14\|s\|_{H^{-1}(\rho)}^2 + \| {-}\cV_\cL(\rho)\|^2_{H^1(\rho)} + \langle \cV_\cL(\rho),s\rangle,
\end{equation*}
where the left-hand side is the large-deviation L-function~\eqref{eq:diffusion L-function}, but the right-hand side can not be interpreted as a generalized gradient flow.

By contrast, if we assume that $F=\grad P$ for some potential $P$, then $\pi(dx) = e^{-P(x)}\,dx$ is an (locally finite) invariant measure for \eqref{eq:diffusion generator}, and detailed balance~\eqref{eq:detailed balance self-adjoint} is always satisfied for this $\pi$. Therefore, by Theorem~\ref{th:time-rev is detailed balance}, the time-symmetry condition~\eqref{eq:time-symmetric L} and the integrability condition $\cV_\cL(\rho)=D\S_\cL(\rho)$ both hold, with entropy functional
\begin{equation*}
  \S_\cL(\rho):=\tfrac12\E_\pi(\rho)=\frac12\int_{\R^d}\!\rho(x)\log\frac{\rho(x)}{\pi(x)}\,dx=\frac12\int_{\R^d}\!\rho(x)\log\rho(x)\,dx + \frac12\int_{\R^d}\!P(x)\rho(x)\,dx.
\end{equation*}
Note that the HWI inequality \cite[Cor.~20.13]{Villani2009} assures that $\S_\cL$ is indeed finite on $\Z$. Putting the parts together, we find:
\begin{equation*}
  \cL(\rho,s)=\frac14\|s\|_{H^{-1}(\rho)}^2 + \|D\S_\cL(\rho)\|^2_{H^1(\rho)} + \langle D\S_\cL(\rho),s\rangle,
\end{equation*}
where the right-hand side is the well-known quadratic Wasserstein-entropy gradient system.

\section{Conclusion and discussion}
\label{sec:discussion}

In the introduction we asked the question ``Does there exist a general connection between large-deviation rate functionals and generalized gradient flows?'' The results of this paper give an affirmative answer to this question, that we split into steps. 

\textit{Step 1: Careful definition.} What the `connection' above means can be interpreted in multiple ways. The most restricted interpretation would be \emph{find $\Psi$, $\Psi^*$, and $\S$ such that }
\begin{equation}
\label{interpretation1}
\cL(z,\dot z) = 0 \quad\Longleftrightarrow \quad \Psi(z,\dot z) + \Psi^*(z,-D\S(z)) + \langle D\S(z),\dot z\rangle=0.
\end{equation}
This interpretation leaves much room: given \emph{any} function $\S$, a function $\Psi$ can be chosen such that $D\Psi^*(z,-D\S(z))=\cA_\cL(z)$, implying~\eqref{interpretation1}. In a Hilbert space an example would be $\Psi^*(z,\xi) := \frac12 \|\xi+D\S(z)-\cA(z)\|^2$.

Instead we require more, namely
\begin{equation}
\label{interpretation2}
\cL(z,\dot z) = \Psi(z,\dot z) + \Psi^*(z,-D\S(z)) + \langle D\S(z),\dot z\rangle.
\end{equation}
This property formulates a relationship between not only the zero values of the left- and right-hand sides, but all values. We comment on the interpretation of the other values below. 

However, this requirement is still too weak to be meaningful. Lemma~\ref{lem:L=Psi Psi*} shows the slightly surprising fact that it is still possible, for \emph{any} choice of $\S$, to construct a pair $\Psi$, $\Psi^*$ such that~\eqref{interpretation2} is satisfied. Therefore even~\eqref{interpretation2} does not carry enough information. 

To improve the situation we require a further property, which is that $\Psi$ and $\Psi^*$ are minimized at zero, which is equivalent to $\Psi$ and $\Psi^*$ being non-negative. This additional condition is inspired by the fact that it makes $\S$ into a Lyapunov function and that stationary points of $\S$ are also stationary points of the evolution. It is with this additional condition that we call $(\Z,\S,\Psi)$ a generalized gradient system, and this is the connection that we seek.

\textit{Step 2: Abstract characterization.}
Now that we have defined what we are looking for, Theorem~\ref{th:L function to generalized gradient system} gives a clear characterization of the possibility of finding it. Given an L-function $\cL$, a corresponding generalized gradient system exists \emph{if and only if }
\begin{equation}
\label{DsL-is-a-derivative}
D_s\cL(z,0) \text{ is a derivative (of a functional $\overline\S$, say)}.
\end{equation}
In that case there is \emph{exactly one} system $(\Z,\S,\Psi)$ satisfying~\eqref{interpretation2} (up to addition of constants in $\S$), and $\S= \overline \S + \text{constant}$. 

At this level it is also possible to characterize when $\Psi$ and $\Psi^*$ are symmetric: if $\cL$ satisfies the relation $\cL(z,s) - \cL(z,-s) = 2\langle D\overline\S(z),s\rangle$ for some functional $\overline\S$, then~\eqref{DsL-is-a-derivative} is automatically fulfilled, and in addition $\Psi$ and $\Psi^*$ are symmetric.

\textit{Step 3: Application to empirical processes.} 
We next specialize to the case when $\cL$ is the large-deviation rate functional of an empirical measure of i.i.d.\ Markov processes. Theorem~\ref{th:time-rev is detailed balance} states that the existence of a generalized gradient system \emph{is equivalent to the property of detailed balance} of the underlying processes. In this case the unique functional~$\S$ is one-half of the relative entropy with respect to the invariant measure. This property is well known in the probability community, in an informal way, but we are not aware of a proof at this level of abstraction. In Section~\ref{sec:applications} we illustrate these abstract concepts on two examples. 

Summarizing, with these results we identify a class of L-functions that have corresponding generalized gradient systems, which coincides with the class of detailed-balance systems in the case of empirical measures of i.i.d.\ Markov processes. Let us continue with a few remarks and questions.

\bigskip

\textit{Interpretation of non-zero  values of $\cL$.}
Although the deterministic evolution described by $\cL(z,\dot z) = 0$ depends only on the level set at value zero of $\cL$, other values of $\cL$ are relevant when $\cL$ is the rate function of a large-deviation principle: $\cL(z,\dot z)$ can be interpreted as a `cost' to deviate from the optimal value $\dot z = \cA(z)$. Relation~\eqref{eq:L=Psi Psi*} shows that this cost is simultaneously characterized in terms of the components of the generalized gradient flow.

If $\cL$ is interpreted as a cost function, then \eqref{interpretation2} shows that this cost consists of two parts: 
\begin{itemize}
  \item the cost not to move at all rather than in the optimal direction is given by $\cL(z,0)-\cL\big(z,\cA(z)\big)=\cL(z,0)=\Psi^*\big(z,-D\S_\cL(z)\big)$;
  \item the net cost to move in direction $s$ rather than not to move at all is given by $\cL(z,s)-\cL(z,0)=\Psi(z,s)+\langle D\S_\cL(z),s\rangle$. %This nett cost, in turn, consists of the cost $\Psi(z,s)$ and the gain $-\langle D\S_\cL(z),s\rangle$.
\end{itemize}
It is because of this splitting that the integrability condition needs to checked for the tangent $s=0$, although it yields a gradient system for the tangent $s=\cA_\cL(\rho)$.

\textit{Deviations from detailed balance.}
In the case of empirical measures \emph{without} detailed balance, certain parts of the properties above remain. 
The existence of a `generalized covector system' $(\Z,\Psi,\cV)$ with non-negative $\Psi$ and $\Psi^*$ follows under similar conditions, as we show in Lemma~\ref{lem:L=Psi Psi*} and Proposition~\ref{prop:time-symmetric L to gradient flow}. Of course such systems do not identify a functional, and therefore the question whether the system is `driven' by some functional $\S$ remains open. 

Interestingly, the relative entropy with respect to any invariant measure $\pi$ is always a Lyapunov function, as is even every convex functional of the form $\rho \mapsto \int f(d\rho/d\pi)\, d\pi$. Therefore there is no shortage of Lyapunov functionals; but at this stage we do not have a `canonical' way of coupling any of these to a given function $\cL$.

In specific cases one can say more. In related work~\cite{Duong2013} we showed that the large-deviation rate function $\cL$ for the empirical measure of inertial particles (even with interaction) can be written as
\[
\cL(z,\dot z) = \Psi(z,\dot z - \mathscr{H}\!z) + \Psi^*(z,-D\S(z)) + \langle D\S(z),\dot z\rangle,
\]
where $\mathscr H\!z$ is the non-detailed-balance part of the limiting evolution resulting from inertia, and $\S$ is a  relative entropy. Indeed this functional $\cL$ does not satisfy the integrability condition~\eqref{DsL-is-a-derivative}. What forms of $\cL$ this example suggests us to expect in more general situations is still unclear.

\medskip
\textit{Multiple gradient-flow structures for finite-state Markov chains.}
One of the original motivations was to understand the origin of the 
gradient structure for time-continuous Markov chains with detailed
balance found
independently by Maas \cite{Maas2011}, Chow et al \cite{Chow2012},
and \cite{Mielke2011a,Mielke2012a}. This gradient structure is a
true gradient structure for the relative entropy in the sense
that the dissipation potential is quadratic. However, it turned out
that the LDP produces a generalized gradient structure. In fact, these
two gradient structures are special cases of a general family where
$\Psi^*$ is given in terms of  general scalar dissipation functions
$\psi_{ij}$ in the form 
\[
\Psi^*(\rho,\xi) = \sum_{i,j=1}^J  L_{ij}(\rho)\psi_{ij}(\xi_j-\xi_i) 
\quad \text{with } L_{ij}(\rho)=\pi_i Q_{ij}\dfrac{\frac{\rho_j}{\pi_j} -
  \frac{\rho_i}{\pi_i}}{\psi'_{ij}\big(\log\frac{\rho_j}{\pi_j} -
  \log \frac{\rho_i}{\pi_i} \big)}  \text{ for }i\neq j,
\]
and $L_{jj} \equiv 0$. Here $L_{ij}(\rho)$ is always non-negative
since $\psi_{ij}$ is strictly convex. The case $\psi_{ij}(\zeta)=\frac12\zeta^2$
leads to the above-mentioned gradient structure, while
$\psi_{ij}(\zeta)= \cosh \zeta -1 $ generates the generalized gradient
structure obtained from the LDP. 
It remains an open question to understand the role of these other gradient
structures and their relation with stochastic
processes.

Another interesting challenge is the stochastic origin of the
gradient structure for general nonlinear reaction systems or even
reaction-diffusion systems as discussed in \cite{Mielke2011a}. There
again a nonlinear version of the detailed balance condition (also
called Wegscheider condition for reaction systems) is crucial
for the existence of the gradient structure with the relative entropy
as the driving functional.   
\EEE

\appendix

\section{Rigorous proof of the large-deviation principle}

In this appendix we prove the large-deviation principle~\eqref{eq:path ldp} for finite-state continuous-time Markov chains (the formal calculation can also be found in \cite[Sec.~4.3]{Wang2012}). For this result we will need the space $BV(0,T;\P({\{1,\hdots,J\}}))$, defined as a subspace of the functions of bounded variation with values in $\R^J$, equipped with the total-variation norm, 
\begin{align*}
  \|\rho\|_{TV} %&:= \sup \left \{\sum_{j=1}^m |\rho_{t_j} - \rho_{t_{j-1}}|_{\R^J}: 0\leq t_0 < t_1 < \dots \leq t_m\leq T\right\}\\
                := \sup\left\{ \int_0^T\!\dot\phi_t\cdot\rho_t\,dt: \phi\in C^1([0,T];\R^J), \|\phi\|_{\infty}\leq1\right\}\!.
\end{align*}
The space $BV(0,T;\R^J)$ can be interpreted as the dual of $C^1([0,T];\R^J)$, if the dual pairing $\int_0^T\dot\phi_t\rho_t\,dt$ is used. We then define $\Theta$ to be the space $BV\big(0,T;\P({\{1,\hdots,J\}})\big)$, endowed with the weak-$*$ topology.

\begin{theorem}
\label{th:Markov ldp}
Consider the Markov processes $X_1(t),X_2(t),\hdots$ and the corresponding counting process $\rho^{(n)}:t\mapsto\sum_{k=1}^n \Indicator_{X_k(t)}$ as defined in Section~\ref{subsec:Markov stochsys}. Assume that the initial states $\rho^{(n)}_0$ are deterministic and converge in $\R^J$ to the invariant measure $\mu$. Then~$\rho^{(n)}$ satisfies a large-deviation principle in the space $\Theta$ with good rate function
\[
I(\rho) := \begin{cases}
\int_0^T \cL(\rho_t,\dot \rho_t)\, dt &\text{if $\rho$ is absolutely continuous,}\\
+\infty &\text{otherwise},
\end{cases}
\]
where $\cL$ is defined in~\eqref{eq:Markov Lagrangian}. 
\end{theorem}

\bigskip

The proof consists of three lemmas. Lemmas~\ref{lemma:lowerbound} and \ref{lemma:upperbound} together prove the `weak' large-deviation principle. The exponential tightness, proven in Lemma~\ref{lemma:exptight}, then guarantees that the (`strong') large-deviation principle is also satisfied~\cite[Lem.~1.2.18]{Dembo1998}.

\begin{lemma}\textbf{(Exponential tightness)}
\label{lemma:exptight}
For all $\alpha>0$ there exists a compact set $K\subset \Theta$ such that 
\[
\Prob(\rho^{(n)}\not\in K) \leq e^{-\alpha n}.
\]
\end{lemma}

\begin{lemma}\textbf{(Lower bound)}
\label{lemma:lowerbound}
For all open sets $O\subset \Theta$, 
\[
\liminf_{n\to\infty} \frac1n \log \Prob(\rho^{(n)}\in O) \geq -\inf_O I.
\]
\end{lemma}

\begin{lemma}\textbf{(Upper bound)}
\label{lemma:upperbound}
For all compact sets $C\subset \Theta$, 
\[
\limsup_{n\to\infty} \frac1n \log \Prob(\rho^{(n)}\in C) \leq -\inf_C I.
\]
\end{lemma}

We now prove these lemmas.
\begin{proof}[Proof of Lemma~\ref{lemma:exptight}, Exponential tightness]
Take an arbitrary $\alpha>0$, and define for some $M>0$, to be chosen later, the set
\[ 
K:= \overline{\{\rho\in \Theta: \|\rho\|_{TV} \leq M\}}.
\]
Here the closure is in the weak-* topology, such that $K$ is automatically compact by the Banach-Alaoglu Theorem.

To estimate $\Prob(\rho^{(n)}\in K^c)$, we first estimate
\begin{align}
\|\rho^{(n)}\|_{TV} &\leq \frac1n \sum_{k=1}^n \|\Indicator_{X_k}\|_{TV}=\frac2n\sum_{k=1}^n \Lambda_k,
%= \frac2n \sum_{k=1}^n \# \{\text{jumps of $X_k$}\} \notag\\
%&= \frac2n \#\{\text{jumps of all $X_k,k=1,\hdots,n$}\}. 
\label{eq:bound TV by jumps}
\end{align}
where $\Lambda_k$ is the number of jumps of $X_k$ in the time interval $[0,T]$. Therefore we estimate the probability of the set $K^c$ by estimating the number of jumps of all particles.

We apply a Chernoff estimate to find
\begin{align*}
\Prob\Big(\frac2n\sum_{k=1}^n\Lambda_k\geq M \Big)&= 
\Prob\Big(\exp\sum_{k=1}^n \Lambda_k \geq \exp \frac{nM}{2}\Big)\\
&\leq e^{-\frac{nM}{2}} \Expectation \Big(\exp \sum_{k=1}^n \Lambda_k\Big)
= e^{-\frac{nM}{2}} \prod_{k=1}^n \Expectation (\exp \Lambda_k),
\end{align*}
by the independence of the $\Lambda_k$. Let $\gamma := \max_{i=1,\hdots,J} \sum_{j=1,j\neq i}^J Q_{ij}$ be the maximal rate of jumping away from any site. Comparing the random variable $\Lambda_k$ with a Poisson process $L$ with rate $\gamma$, we find that
\[
\Expectation (\exp \Lambda_k)\leq \Expectation(\exp L_T) = e^{\gamma T(e-1)}.
\]
Then, continuing the calculation above, 
\[
\Prob\Big(\frac2n\sum_{k=1}^n \Lambda_k\geq M\Big)\leq
e^{-\frac{nM}{2}} \prod_{k=1}^n e^{\gamma T(e-1)}
= \exp -n\big(\tfrac{M}{2} - \gamma T(e-1)\big).
\]
By choosing $M>2\big(\alpha + \gamma T(e-1)\big)$ we have 
\[
\Prob(\rho^{(n)}\not\in K)\leq\Prob\big(\|\rho^{(n)}\|_{TV}\geq M\big) \mathop{\leq}^{\eqref{eq:bound TV by jumps}} \Prob\bigg(\sum_{k=1}^n\Lambda_k\leq \frac{nM}{\sqrt{2}}\bigg)\leq e^{-\alpha n}
\qquad \text{for all }n.
\]
This proves the assertion.
\end{proof}

For the upper and lower bound it will be useful to have a dual characterization of the proposed rate function $\int_0^T\cL$ similar to the dual characterization of $\cL$ itself. For $\rho\in \Theta$ and $\xi\in C_b([0,T];\R^J)$, define 
\[
G(\rho, \xi) := \int_0^T \Big[ \dual{\xi_t}{\dot\rho_t} - \cH(\rho_t,\xi_t)\Big]\, dt.
\]
Here and below we use the notation $\xi_t$ for the value of $\xi$ at time $t$, and $\xi_t(i)$ for the $i$-coordinate of the vector $\xi_t\in \R^J$, for $i\in {\{1,\hdots,J\}}$. Moreover, to simplify notation we write $\langle \xi,\rho\rangle:=\sum_{i=1}^J \xi(i)\rho(i)$.

\begin{lemma}
For all $\rho\in \Theta$, 
\begin{equation}
\label{eq:lemma:IG}
I(\rho) = \sup\left\{ G(\rho,\xi): \xi\in C_b([0,T];\R^J)\right\}.
\end{equation}
\end{lemma}

\begin{proof}
The proof is standard, and proceeds in three steps:
\begin{enumerate}
\item The inequality ``$\geq$" follows directly from the duality relation~\eqref{eq:Markov Lagrangian}. 
\item The implication ``$\rho$ not absolutely continuous $\Longrightarrow$ the supremum is $+\infty$'' follows from an explicit choice of subintervals in which $|\dot \rho|$ is large; by concentrating $\xi$ on those subintervals the supremum can be made as large as necessary.
\item Finally,  the inequality ``$\leq$'' follows by a standard approximation argument, in which~$\xi$ is constructed to approximate the pointwise optimum in the duality~\eqref{eq:Markov Lagrangian}.
\end{enumerate}
\end{proof}

We also need a Girsanov-type theorem that characterizes the Radon-Nikodym derivative of the original process with respect to a modified process. It is no coincidence that the functional $G$ defined above arises as in this characterization. To formulate this result we first introduce some notation.
Let $\bP$ be the law of a single process $X_k(t)$, generated by $\Q$, and for any $\xi\in C^1([0,T];\R^J)$, let $\bP_f$ be the law of a single process generated by the modified, time-dependent generator
\[
  (\Q_{\xi}(t)g)_i := \sum_{j=1}^J Q_{ij}e^{\xi_t(j)-\xi_t(i)}\big(g(j)-g(i)\big), \qquad\text{for }g\in\R^J, i\in {\{1,\hdots,J\}}.
\]
We write $\bP^{(n)}$ and $\Expectation^{(n)}$ for the law and expectation of the empirical process $\rho^{(n)}$, and $\bP^{(n)}_\xi$ and $\Expectation^{(n)}_\xi$ for the law and expectation of the modified empirical process.

\begin{lemma}\text{(A Girsanov-type characterization)}
\label{lemma:Girsanov}
For $\bP^{(n)}$-almost every $\rho$ of the form $t\mapsto\frac1n \sum_{k=1}^n \Indicator_{x_k(t)}$,
\[
\frac1n \log \frac{\mathrm d\bP^{(n)}_\xi}{\mathrm d\bP^{(n)}}(\rho) 
= G(\rho,\xi).
\]
\end{lemma}

\begin{proof}
Proposition~7.3 of the Appendix~1 of~\cite{Kipnis1999} gives that for any c\`adl\`ag function $x:[0,T]\to{\{1,\hdots,J\}}$, 
\begin{equation}
  \log \frac{\mathrm d\bP_\xi}{\mathrm d\bP}(x) = 
  \xi_T(x_{T-}) - \xi_0(x_{0}) - \int_0^T e^{-\xi_t(x_t)}\Big[(\partial_t + \Q)e^{\xi_t}\Big](x_t)\, dt. \label{eq:app nonlinear generator}
\end{equation}
We calculate for $i\in{\{1,\hdots,J\}}$ that (using convention~\eqref{eq:Markov matrix})
\begin{equation*}
e^{-\xi_t(i)}(\Q e^{\xi_t})(i) = \sum_{j=1}^J Q_{ij}e^{\xi_t(j)-\xi_t(i)}=\sum_{j=1}^J Q_{ij}\big(e^{\xi_t(j)-\xi_t(i)}-1\big)= \cH\big(\Indicator_{x},\xi_t\big),
\end{equation*}
and $e^{-\xi_t}\partial_t\big(e^{\xi_t}\big) = \dot \xi_t$.
Therefore, for a c\`adl\`ag function $x:[0,T]\to{\{1,\hdots,J\}}$,
\begin{align*}
\log \frac{\mathrm d\bP_\xi}{\mathrm d\bP}(x)&=
\xi_T(x_{T-}) - \xi_0(x_{0}) - \int_0^T e^{-\xi_t(x_t)}\big(\partial_t e^{\xi_t}\big)(x_t)\, dt
- \int_0^T \cH(\Indicator_{x_t},\xi_t)\, dt\\
&= \dual{\xi_T}{\Indicator_{x_{T-}}} - \dual{\xi_0}{\Indicator_{x_{0}}} - \int_0^T \dual{\dot \xi_t}{\Indicator_{x_t}}\, dt
- \int_0^T \cH(\Indicator_{x_t},\xi_t)\, dt.
\end{align*}
Since the time evolution of the $n$ processes is independent, and the Radon-Nikodym derivates below are invariant under the push-forward $(x_1,\hdots,x_n)\mapsto\frac1n \sum_{k=1}^n \Indicator_{x_k}=:\rho$, we have
\begin{align*}
\frac1n \log \frac{\mathrm d\bP^{(n)}_\xi}{\mathrm d\bP^{(n)}}(\rho)&=
\frac1n \log \prod_{k=1}^n\frac{\mathrm d\bP_\xi}{\mathrm d\bP}(x_k)
= \frac1n \sum_{k=1}^n\log \frac{\mathrm d\bP_\xi}{\mathrm d\bP}(x_k)\\
&= \dual{\xi_T}{\rho_{T-}} - \dual{\xi_0}{\rho_{0}} - \int_0^T \dual{\dot \xi_t}{\rho_t}\, dt
- \int_0^T \cH(\rho_t,\xi_t)\, dt.
\end{align*}
With probability one, each process only jumps a finite number of times, and therefore $\bP^{(n)}$-almost every $\rho$ is piecewise constant, with a finite number of jumps. Therefore the time derivative $\dot \rho$ is a finite measure, so that we can apply partial integration to obtain
\[
\frac1n \log \frac{\mathrm d\bP^{(n)}_\xi}{\mathrm d\bP^{(n)}}(\rho)
= \int_0^T \Big[\dual{\xi_t}{\dot \rho_t}- \cH(\rho_t,\xi_t)\Big]\, dt
= G(\rho,\xi).
\]
\end{proof}

\bigskip

We now continue with the proofs of the upper and lower bounds. 

\begin{proof}[Proof of Lemma~\ref{lemma:lowerbound}, the lower bound]
Fix an open set $O\subset \Theta$ and $\delta>0$. Choose $\rho\in \Theta$, smooth as a function from $[0,T]$ into $\R^J$, such that 
\[
I(\rho)\leq \delta +\inf_O I .
\]
Define $\xi\in C^1([0,T];\R^J)$ to achieve the supremum in~\eqref{eq:lemma:IG}.

Choose $\e>0$ such that the $\R^J$-ball $B(\rho,\e)$ is contained in $O$, so that 
\begin{align}
\Prob(\rho^{(n)}\in O) & \geq \Prob\big(\rho^{(n)}\in B(\rho,\e)\big) = \Expectation^{(n)}(\Indicator_{B(\rho,\e)}) \notag\\
  &= \Expectation^{(n)}_\xi\left(\Indicator_{B(\rho,\e)}\frac{\mathrm d\bP^{(n)}}{\mathrm d\bP^{(n)}_\xi}\right)
\geq \left(\mathop{\hbox{$\bP^{(n)}_\xi$-ess\,inf}}\limits_{B(\rho,\e)} \frac{\mathrm d\bP^{(n)}}{\mathrm d\bP^{(n)}_\xi}\right) \label{eq:ldp lower bound est1}
\Expectation^{(n)}_\xi\left(\Indicator_{B(\rho,\e)}\right).
\end{align}
Since $\xi$ achieves the optimum in~\eqref{eq:lemma:IG}, we have for all $g$ that 
\begin{align*}
0 &= \langle D_\xi G(\rho,\xi),g\rangle \\
  &= \int_0^T\! \big\{ \dual{g_t }{\dot \rho_t}
     - \dual{D_\xi \cH(\rho_t,\xi_t)}{g_t}\big\}\,dt\\
  &= \int_0^T\! \bigg\{ \dual{g_t }{\dot \rho_t}
     - \sum_{i,j=1}^J \rho_t(i)Q_{ij}e^{\xi_t(j)-\xi_t(i)}\big(g_t(j)-g_t(i)\big)\bigg\}\,dt\\
  &=\int_0^T\! \big\{ \dual{g_t }{\dot \rho_t}
  - \dual{\Q_\xi(t)g_t}{\rho_t}\big\}\,dt.
\end{align*}
From this it follows that $\xi$ is chosen such that $\Q_\xi$ generates the process described by $\rho$. Therefore, by the law of large numbers, 
\begin{equation}
  \Expectation^{(n)}_\xi \left(\Indicator_{B(\rho,\e)}\right) \longrightarrow 1 \qquad\text{as }n\to\infty. \label{eq:ldp lower bound est2}
\end{equation}

Since $\bP^{(n)}_\xi$-almost every $\hat \rho\in B(\rho,\e)$ is of the form $\hat \rho = \frac1n \sum_{k=1}^n\Indicator_{x_k}$, Lemma~\ref{lemma:Girsanov} gives
\begin{align}
\frac1n \log \frac{\mathrm d\bP^{(n)}}{\mathrm d\bP^{(n)}_\xi}(\hat \rho) 
&= -G(\hat \rho,\xi) = -G(\rho,\xi) + G(\rho-\hat\rho,\xi) \notag \\
&\geq -I(\rho) - \sup_{\hat\rho\in B(\rho,\e)} |G(\rho-\hat\rho,\xi)| \notag\\
&\geq -\inf_O I - \delta - \sup_{\hat\rho\in B(\rho,\e)} |G(\rho-\hat\rho,\xi)|. \label{eq:ldp lower bound est3}
\end{align}
Here, the last term converges to zero as $\e\to0$, since for fixed $\xi$ the mapping $\rho\mapsto G(\rho,\xi)$ is continuous.
Taking the limit first as $\e\to0$ and then $\delta\to 0$ the claim follows from \eqref{eq:ldp lower bound est1},\eqref{eq:ldp lower bound est2} and \eqref{eq:ldp lower bound est3}. 
\end{proof}

\begin{proof}[Proof of Lemma~\ref{lemma:upperbound}, the upper bound]
Fix a compact set $C\subset\Theta$. First we show that
\begin{equation}
\label{est:upperboundintermediate}
\limsup_{n\to\infty} \frac1n \log \Prob(\rho^{(n)}\in C)\leq
-\sup_{\cup_{k} C_k\supset C} \min_k \sup_\xi \inf_{\rho\in C_k}
G(\rho,\xi),
\end{equation}
where the first supremum is taken over all finite coverings $\{C_k\}_{k=1}^K$ of $C$.

To prove this inequality, first note that for any such covering, $\Prob(\rho^{(n)}\in C)\leq \sum_k\Prob_n(\rho^{(n)}\in C_k)$, and therefore
\begin{equation}
  \limsup_{n\to\infty} \frac1n \log \Prob(\rho^{(n)}\in C)\leq \max_k \limsup_{n\to\infty} \frac1n \log \Prob(\rho^{(n)}\in C_k).
\label{eq:ldp upper bound est1}
\end{equation}
We then estimate, using Lemma~\ref{lemma:Girsanov},
\begin{equation}
\begin{split}
\Prob(\rho^{(n)}\in C_k) = \Expectation^{(n)}(\Indicator_{C_k}) 
= \Expectation^{(n)}_\xi\left(\Indicator_{C_k}\frac{\mathrm d\bP^n}{\mathrm d\bP^{(n)}_\xi}\right)
&\leq \mathop{\hbox{$\bP^{(n)}_\xi$-ess\,sup}}\limits_{C_k} \frac{\mathrm d\bP^{(n)}}{\mathrm d\bP^{(n)}_\xi}\\
&\leq  -\mathop{\hbox{$\bP^{(n)}_\xi$-ess\,inf}}\limits_{C_k} G(\cdot,\xi)
\leq -\inf_{C_k} G(\cdot,\xi).
\end{split}
\end{equation}
Taking the limit $n\to\infty$ and the infimum over $\xi$, we find
\[
\limsup_{n\to\infty} \frac1n \log \Prob(\rho^{(n)}\in C_k) \leq -\sup_\xi \inf_{C_k} G(\cdot,\xi).
\]
Therefore, \eqref{eq:ldp upper bound est1} can be further estimated by
\[
\limsup_{n\to\infty} \frac1n \log \Prob(\rho^{(n)}\in C) \leq -\max_k\sup_\xi \inf_{C_k} G(\cdot,\xi),
\]
and since this holds for all coverings $\{C_k\}$ the inequality~\eqref{est:upperboundintermediate} follows.

\bigskip

We now show that 
\begin{equation}
\label{ineq:minmax}
\sup_{\cup_{k} C_k\supset C} \min_k \sup_\xi \inf_{\rho\in C_k}
G(\rho,\xi)
\geq \inf_{\rho\in C} \sup_\xi G(\rho,\xi) = \inf_C I,
\end{equation}
which, together with \eqref{est:upperboundintermediate} will conclude the proof. Naturally, the last equality follows from~\eqref{eq:lemma:IG}. To show the inequality in \eqref{ineq:minmax}, fix a $\delta>0$ and let $\eta := \inf_{\rho\in C} \sup_\xi G(\rho,\xi)$. Observe that for each $\rho$ there exists an $\xi_\rho$ such that $G(\rho,\xi_\rho)\geq \eta-\delta$. Since $\rho\mapsto G(\rho,\xi)$ is continuous for any $\xi$, it follows that 
\[
\forall \rho\in C\;\exists \xi_\rho \; \exists O_\rho\ni \rho \text{ open such that }
\inf_{O_\rho} G(\cdot,\xi_\rho) \geq \eta-2\delta,
\]
from which we deduce that
\begin{equation}
\label{ineq:supinfG}
\forall \rho: \quad \sup_\xi \inf_{O_\rho} G(\cdot,\xi)\geq \eta-2\delta.
\end{equation}
Since $C$ is compact and $C\subset \bigcup_{\rho\in C}O_\rho$, there is a finite covering $C\subset \bigcup_{k=1}^K O_{\rho_k}$; by~\eqref{ineq:supinfG} we have
\[
\min_k \sup_\xi \inf_{O_{\rho_k}} G(\cdot,\xi)\geq \eta-2\delta.
\]
Trivially, we have
\[
\sup_{\cup_{k} C_k\supset C} \min_k \sup_\xi \inf_{O_{\rho_k}} G(\cdot,\xi)\geq \eta-2\delta.
\]
Since $\delta>0$ is arbitrary, the inequality~\eqref{ineq:minmax} follows.
\end{proof}

\bibliographystyle{alpha}
\bibliography{library}
\end{document}